\documentclass[a4paper,reqno]{amsart}
 \usepackage{amssymb,amsmath,amscd,graphicx,color,epstopdf,mathtools}
\usepackage[all]{xy}
\usepackage[all]{xy}
\usepackage{color}
\usepackage{hyperref}
\usepackage{enumerate}
\usepackage{amsthm}
\usepackage{epsfig}
\usepackage[english]{babel}

\newcommand{\so}{\mathfrak{so}}
\newcommand{\pp}{{\mathfrak{p}}} 
\newcommand{\nn}{{\mathfrak{n}}}

\definecolor{darkgreen}{rgb}{0.0, 0.5, 0.0}

\newcommand{\lin}{{\mathrm{lin}}} 
\renewcommand{\ss}{{\mathfrak{s}}}

  \newcommand{\w}{\omega}
  
\newcommand{\R}{\mathbb{R}}

\newcommand{\C}{\mathbb{C}}

\newcommand{\Z}{\mathbb{Z}}

\newcommand{\SO}{\mathrm{SO}}
  
\newcommand{\GG}{\mathrm{G}}
\newcommand{\KK}{\mathrm{K}}

     \newcommand{\cA}{\mathcal{A}}
     \newcommand{\cB}{\mathcal{B}}

  \newcommand{\cO}{\mathcal{O}}

\renewcommand{\O}{\mathcal{O}}             
\renewcommand{\gg}{\mathfrak{g}}        
\newcommand{\hh}{\mathfrak{h}}          
\newcommand{\kk}{\mathfrak{k}}          
\renewcommand{\tt}{\mathfrak{t}}        
  \renewcommand{\aa}{\mathfrak{a}}          
\renewcommand{\sl}{\mathfrak{sl}}   
\newcommand{\SL}{\mathrm{SL}}
\newcommand{\SU}{\mathrm{SU}}

\numberwithin{equation}{section}
\allowdisplaybreaks

\newtheorem{theorem}{Theorem}[section]
\newtheorem{lemma}[theorem]{Lemma}
\newtheorem{proposition}[theorem]{Proposition}

\theoremstyle{definition}

\newtheorem{example}[theorem]{Example}

\newtheorem{remark}[theorem]{Remark}

\begin{document}

\author{David Mart\'inez Torres}
\address{Applied Mathematics Department, ETSAM Section, Universidad Polit\'ecnica de Madrid, 
Avda. Juan de Herrera, 4, 28040, Madrid, Spain}
\email{df.mtorres@upm.es}

\title{Canonical domains for coadjoint orbits}

\begin{abstract}
This paper describes two real analytic symplectomorphisms defined on appropriate dense open subsets
of any coadjoint orbit of a compact semisimple Lie algebra. The first symplectomorphism sends the open dense
subset to a bounded subset of a standard cotangent bundle. The second symplectomorphism has as target a bounded subset 
of a hyperbolic coadjoint orbit of an associated non-compact semisimple Lie algebra. Therefore, coadjoint orbits 
of compact Lie algebras are symplectic compactifications of domains of cotangent bundles, and are in symplectic
correspondence with hyperbolic orbits of non-compact semisimple Lie algebras.
\end{abstract}

\maketitle


\section{Introduction}

Let $\gg$ be a compact semisimple Lie algebra. The Killing form intertwines the coadjoint and adjoint representations of 
$\gg$ on $\gg^*$ and $\gg$, respectively. 
Therefore, any adjoint orbit is equipped with the Kostant--Kirillov--Souriau symplectic form for which 
the adjoint action of $\gg$ is Hamiltonian.

Let $\GG$ be a connected Lie group whose Lie algebra is $\gg$. If $X$ is an adjoint orbit, then  it can be endowed with a complex structure (a K\"ahler structure with K\"ahler form the Kostant-Kirillov-Souriau symplectic form) such that the complexification
of $\GG$ acts on $X$ by holomorphic transformations, and thus so do any of its subgroups.  A group involution 
on $\GG$ determines one such subgroup upon complexification of its fixed point set. This is a so-called spherical subgroup:
Its action on $X$ has a finite number of orbits \cite{Sp,Ma}. In particular, it has a dense open orbit $X^*\subset X$.

Our first result in this paper concerns the existence of canonical coordinates for the restriction of the $\mathrm{KKS}$ symplectic form to $X^*$
for appropriate choices of involution.

\begin{theorem}\label{thm:canonical} Let $(X,\w)$ be an adjoint orbit of the compact semisimple Lie algebra $\gg$ endowed with its
$\mathrm{KKS}$ symplectic form. A Lie algebra involution whose opposite involution $\sigma$ acts on $X$ with non-empty fixed point set
determines a real analytic diffeomorphism
\begin{equation}\label{eq:canonical}
 \psi:(X^*,\w,Y,\sigma)\to (D\subset T^*L,d\lambda,E,\iota),
\end{equation}
where
\begin{itemize}
\item $Y$ is a Liouville vector field on $X^*$;
\item $L\subset X^*$ is the fixed point set of $\sigma$ on $X$ (and therefore it is a Lagrangian submanifold);
\item $D$ is an open bounded star-shaped neighborhood of the zero section of $T^*L$;
\item $d\lambda$ is the Liouville symplectic form of $T^*L$;
\item $E$ is the Euler vector field of  $T^*L$;
\item $\iota$ is the involution on $T^*L$ which sends a covector to its opposite.
\end{itemize}

The real analytic diffeomorphism in (\ref{eq:canonical}) is $\KK$-equivariant and compatible with the canonical (linear) momentum maps
for the $\KK$-action, where $\KK$ is the connected integration of the fixed point set of the involution. 
\end{theorem}
Theorem \ref{thm:canonical} describes any symplectic (co)adjoint orbit of a compact Lie algebra as a compactification of a domain
in a cotangent bundle with its Liouville symplectic form.

To any Lie algebra involution on $\gg$, there corresponds a non-compact semisimple Lie algebra $\hh$, which is a real form of the complexification of $\gg$.
If the opposite involution 
acts on the adjoint orbit $X\subset \gg$ with non-empty fixed point set, then to $X$ there corresponds
an adjoint orbit $X^\vee$ of $\hh$ and a holomorphic adjoint orbit $\cO$ of the complexified Lie algebra, such that 
\[X\subset \cO \supset X^\vee.\]
If $\Omega$ denotes the holomorphic KKS form on $\cO$, then the real KKS symplectic forms on  $X$ and $X^\vee$ are $-\Im \Omega$ and $\Re\Omega$, respectively.
The real hyperbolic orbit $(X^\vee,\Re\Omega)$ is canonically symplectomorphic to $(T^*L,d\lambda)$ \cite{ABB,Li,Ma}. Therefore, in view of Theorem \ref{thm:canonical}, it
is natural to ask whether it is possible to produce a transformation of the holomorphic adjoint orbit $\cO$, which takes $(X^*,-\Im\Omega)$ 
to a domain of $(X^\vee,\Re\Omega)$.

Our second result describes such a symplectomorphism by means of a Wick rotation-type map:

\begin{theorem}\label{thm:duality} Let $(X,\w)$ be an adjoint orbit of the compact semisimple Lie algebra $\gg$ endowed with its
$\mathrm{KKS}$ symplectic form. A Lie algebra involution whose opposite involution $\sigma$ acts on $X$ with non-empty fixed point set
determines inclusions of symplectic manifolds
 \[\xymatrix{ 
  & (\cO,\Omega)   &  \\
(X,\w=-\Im \Omega) \ar[ur]^-{}  &   &     (X^\vee,\Re\Omega)   \ar[ul]^-{}     
}
\]
and a real analytic diffeomorphism 
\begin{equation}\label{eq:duality}
 \Psi:(X^*,\w,Y,\sigma)\to (D^\vee\subset X^\vee,\Re\Omega,\Re\Lambda,-\theta),
\end{equation}
where
\begin{itemize}
\item $D^\vee$ is an open bounded domain of $X^\vee$,
\item $\Lambda$ is the complexification of minus the Liouville vector field  $Y$ in Theorem \ref{thm:canonical},
\item $\theta$ is the Cartan involution on the complexification of $\gg$ which fixes $\gg$,
\end{itemize}
and $\Psi$ is induced by the flow  of $\Lambda$  for time $i\tfrac{\pi}{2}$ .

The real analytic diffeomorphism in (\ref{eq:duality}) is $\KK$-equivariant and compatible with the canonical (linear) momentum maps
for the $\KK$-action, where $\KK$ is the connected integration of the fixed point set of the involution on $\gg$. 

\end{theorem}
Theorem \ref{thm:duality} describes a symplectic correspondence of sorts  between (co)adjoint orbits of a compact Lie algebra
 and hyperbolic (co)adjoint orbits of non-compact real forms of the complexification of the compact Lie algebra. Here, the 
 word \emph{correspondence} should be understood as a bijection between orbits of a pair of different subgroups of a complex Lie group,
 of which there are important examples in the literature with which our construction bears a resemblance; the adjective \emph{symplectic} is added because orbits in correspondence come with a symplectomorphism defined on appropriate domains.  
 \begin{itemize}
  \item The first such instance is the Kostant-Sekiguchi correspondence \cite{Se,V}, which 
 establishes a bijection between  the (finitely many)  nilpotent adjoint  orbits in the real form  $\hh$ with Cartan decomposition $\kk\oplus \ss$, and the nilpotent orbits in the complexification of $\ss$ for the action of the complexification of the maximal compact subgroup $\KK\subset \mathrm{H}$, where $\mathrm{H}$ is the connected integration of $\hh$. Orbits in correspondence 
 lie in the same adjoint of the complexification of $\gg$ and are $\KK$-equivariantly diffeomorphic. In Theorem \ref{thm:duality}, the correspondence is between (some) elliptic orbits and (some) hyperbolic orbits lying in the same orbit of the complexification of $\gg$, and the $\KK$-equivariant symplectomorphism is between domains of the corresponding orbits.
 \item The second example is the Matsuki correspondence \cite{Ma,BL},
 which establishes a bijection between the (finitely many) orbits in $X$ of the action of the complexification of $\KK$ and of $\mathrm{H}$\footnote{The correspondence is more general since it applies to any pair of commuting anti-complex involutions in the complexification of $\gg$. We are assuming that one of them is the Cartan involution so we can establish a relation with Theorem \ref{thm:duality}.}. The correspondence reverts the partial order defined  by the inclusion and orbits in correspondence intersect in a unique $\KK$-orbit. In Theorem \ref{thm:duality}, the open dense subset $X^*$ is the 
 largest orbit of the complexification of $\KK$ and $L=X\cap X^\vee$ is the smallest $\mathrm{H}$-orbit, and thus they are in correspondence. We consider the action of $\mathrm{H}$ in the larger complex orbit $\cO$ --- which induces the action of $\mathrm{H}$ in $X$ --- so that we can realize $X^*$ as a domain of $X^\vee$ in a symplectic and $\KK$-equivariant fashion.
\end{itemize}

The precise meaning of having $\Psi$ in Theorem \ref{thm:duality} induced by the flow of $\Lambda$ is the following:
We cannot grant that the trajectories  starting at all points in $X^*$ exist for time  $i\tfrac{\pi}{2}$. However, they exist for points close
enough to $L$ and we are able to prove that the ensuing map has an analytic continuation to the whole $X^*$. 

Theorem \ref{thm:canonical} generalizes, among others, classical results for projective spaces and quadrics.
If in $\mathfrak{su}(2)$ we consider the involution that conjugates the coefficients of a matrix, then 
Theorem \ref{thm:canonical}  amounts to the following
statement whose origin can be traced back to Archimedes: 
Let $X\subset \R^3$ be a sphere centered at the origin endowed with the Euclidean area form. If we remove the poles, then we obtain 
a subset $X^*$ where we can introduce cylindrical coordinates $\theta,z$. In cylindrical coordinates the Euclidean area form is $d\theta\wedge dz$.

To identify the spherical group of which $X^*$ is an open dense orbit, the sphere, or, rather,  
an adjoint orbit of $\mathfrak{su}(2)$, is identified 
with the complex projective line  mapping a matrix to the eigen-line of the the eigen-value of maximal norm. 
If the equator of the sphere corresponds to 
symmetric matrices, then north and south poles go to the two points of the quadric $\mathcal{Q}_0=\{Z_1^2+Z_2^2=0\}$.
These are the isotropic lines for the standard complex quadratic form in the complex plane. Its symmetry group $\mathrm{SO}(2,\C)\subset \SL(2,\C)$
is the spherical subgroup with which we are concerned. It acts  on the projective line with three orbits: the two points of the quadric and the open orbit that corresponds to 
$X^*$. More generally, for all $n>2$, the action of 
 $\mathrm{SO}(n,\C)$ on $\mathbb{CP}^{n-1}$ has two orbits: the 
quadric  and its complement. Complex conjugation has fixed point set $\mathbb{RP}^{n-1}$. 
The complement of the quadric endowed with the 
Fubini--Study symplectic form is known to be symplectomorphic to a disk bundle of $T^* \mathbb{RP}^{n-1}$ with its Liouville symplectic structure
\cite{BJ,Bir}. This is yet another instance of Theorem \ref{thm:canonical} applied to $\mathfrak{su}(n)$ and 
the involution that conjugates the coefficients of 
a matrix.

Let $X$ be  the product of two spheres and let $X^*$ be the complement of the diagonal. Another classical result in symplectic topology says that 
if $X$ is  endowed with a monotone symplectic form, 
then the anti-diagonal is a Lagrangian 2-sphere and $X^*$ is symplectomorphic to a disk bundle of $T^*S^2$. The product of 
two spheres can be identified with the quadric $\mathcal{Q}_2$ in such a way that the diagonal is mapped to $\mathcal{Q}_1$. More generally,
 if the quadric $\mathcal{Q}_{n-1}$ is endowed with the Fubini--Study form, then the complement of $\mathcal{Q}_{n-2}$ is known to 
 be symplectomorphic to a disk bundle of the Lagrangian $n-1$-sphere obtained as the intersection of $\mathcal{Q}_{n-1}$ with $\mathbb{RP}^{n}$
 \cite{Bir,Au,OU}.  This is another instance of Theorem \ref{thm:canonical} applied to $\so(n,\R)$ and the involution whose associated
 real form is $\so(n-1,1)$.
 
For some orbits, Theorem \ref{thm:canonical} can produce more than one canonical domain. A regular adjoint orbit $X\subset \mathfrak{su}(3)$ is 
identified with the manifold of full flags in $\C^3$. 
Theorem \ref{thm:canonical} applied to the involution that conjugates the coefficients of a matrix produces a canonical domain symplectomorphic to a subset of the cotangent bundle of the manifold of full real 
flags (diffeomorphic to the quotient of $\mathrm{SU}(2)$ by the quaternions). If the eigen-values of $X$ are of the form $-i\lambda,0,i\lambda $, then 
 the involution whose associated real form is $\mathfrak{su}(2,1)$ is in the the hypotheses of 
Theorem \ref{thm:canonical}, and produces a symplectomorphism known to experts in toric symplectic topology: The Gelfand-Zeitlin map has as its only singular
fiber a Lagrangian 3-sphere $S^3$. There is a symplectomorphism to a star-shaped neighborhood of the zero section of $T^*S^3$ defined in the complement of the preimage 
of the two facets of the Gelfand-Zeitlin polytope which do not contain the singular vertex.

The Lagrangians in Theorem \ref{thm:canonical} are well--known to be minimal homogeneous Lagrangians \cite{BG} (see also \cite{My}). We are not aware of a previous 
systematic study of their Weinstein neighborhoods.
The techniques and strategy to prove Theorem \ref{thm:canonical} differ much from those used in \cite{Bir,Au,OU}.
First, Lie theoretic methods are used extensively to describe the 
symplectic geometry of hyperbolic orbits of real and complex semisimple Lie algebras and their relation to Lagrangian fibrations and their affine
geometry \cite{Li}. Second,  
  geometric counterparts of K\"ahler potentials   are employed \cite{Do}; the use of K\"ahler methods, as opposed to projective ones, is important
to have a result that places no integrality requirements on the cohomology class of the coadjoint orbit. Having geometric control on K\"ahler potentials
is fundamental as this translates into geometric control of the associated Liouville vector fields. Third, appropriate Liouville vector fields are 
the key to apply a classical differential geometric characterization of cotangent bundles with their Liouville symplectic form \cite{Na}.

We are not aware of previous instances of Theorem \ref{thm:duality}. In regard to the tools in Theorem \ref{thm:duality}, the use of holomorphic vector fields on K\"ahler manifolds and specifically
on complex (co)adjoint orbits is not new. They are employed in \cite{BBLU} to study the space of K\"ahler structures and the key players there are
holomorphic Hamiltonian vector fields rather than Liouville ones as in Theorem \ref{thm:duality}. In regard to the 
results in Theorem  \ref{thm:duality},  a symplectic variation of the Borel embedding realizing the holomorphic duality between non-compact and compact Hermitian symmetric spaces is discussed in \cite{DL}. This so-called symplectic duality differs much from Theorem \ref{thm:duality} both in the methods and in the result, which is an embedding 
from the non-compact Hermitian symmetric space into its compact dual, which is not a symplectomorphism, but rather relates the KKS symplectic forms by means of a third symplectic form.

There are natural questions to be addressed in relation to Theorem \ref{thm:canonical}. For instance,
at the level of differential topology, it would be important to understand how the domain $D$ is compactified to produce
the coadjoint orbit; it is natural to expect --- at least in some cases --- that the closure of $D\subset T^*L$ be a closed disk bundle and that
there is a Lie group action in the sphere bundle responsible for the identifications. This is the case for projective spaces and quadrics, even 
accounting for the behavior of the symplectic form (cf. \cite{OU}).
At the symplectic level, a relevant question is to describe the shape of $D$ with respect to the cotangent lift of the appropriate bi-invariant metric
in the Lagrangian. This would be the key to have a unified approach to the computation of a bound from below for the Gromov width of coadjoint orbits \cite{FLM} as it would 
reduce the question to estimating embeddings of flat balls in $T^*\R^n$ into $D$. The relation of Theorem \ref{thm:canonical} with 
complete integrable systems should also be addressed. For instance, the compatibility with the Gelfand--Zeitlin system 
should be analyzed in detail (see Example \ref{ex:flags}); to do that, one would not just look at the chain of subalgebras but rather at 
an appropriate chain of commuting involutions. In the other direction, it would be natural to investigate
if some of the known integrable systems on cotangent bundles \cite{My,BJ} can be transplanted
via Theorem \ref{thm:canonical} to integrable systems on coadjoint orbits. It would also be relevant 
to compare Theorem \ref{thm:canonical} with the constructions of large toric charts in \cite{ABHLL} and in \cite{HL}.

The relation of Theorem \ref{thm:duality} with the representation theory of non-compact semisimple Lie algebras should also be addressed.

The organization of this paper is the following. In Section \ref{sec:complexifications}, we discuss how the 
Iwasawa decomposition and anti-complex involutions can be used to describe the geometry of a hyperbolic orbit of a complex semisimple Lie algebra.
Section \ref{sec:LS} starts describing the interaction between the aforementioned hyperbolic orbits, 
holomorphic affine Lagrangian fibrations, and geometric counterparts of K\"ahler potentials for the KKS symplectic structure on an orbit of a compact
Lie algebra. Then, it proceeds to the proof of Theorem \ref{thm:canonical} by combining the previous results with a characterization of (star-shaped domains in)
cotangent bundles described in \cite{Na}. Section \ref{sec:wick-rotation} analyzes the properties of the complexification of the vector field $Y$
in Theorem \ref{thm:canonical} and proceeds to the proof of Theorem \ref{thm:duality}. The most delicate aspect is to control the domain of definition
of the complexified vector field. This is done  by means of geometric methods and its proof is deferred to the Appendix.

We would like to thank R. Vianna for discussions on this matter and the anonymous referee for his/her most valuable comments and suggestions. 

This work was supported by the FAPERJ grants E‐26/202.923/2015, \newline E-
26/010.001249/2016, E‐26/202.908/2018  and the CNPQ grant 304049/2018-2.
 
\section{Complexifications, real forms, and the spherical subgroup}\label{sec:complexifications}

We fix once and for all a compact connected semisimple Lie group $\GG$ whose Lie algebra is $\gg$. For convenience we assume $\GG$ to be simply connected. 
Let $X\subset \gg$ be an adjoint orbit of $\GG$.
It will be convenient to regard $X$ as a real form of a holomorphic adjoint orbit of the complexified Lie algebra $\gg^\C$. 
We denote by $\theta$ the Cartan involution on $\gg^\C$ whose fixed point set is $\gg$. We also fix $\sigma$
an anti-complex involution on $\gg^\C$ that commutes with $\theta$. For simplicity, 
we will assume that its fixed point set is a split real form $\hh$, that is, that $\sigma$ is a Weyl involution.
In Section \ref{ssec:involutions}, we will discuss the case of arbitrary real forms of $\gg^\C$.
We shall use the 
same notation for involutions on the Lie algebra and for their integration on the  complexification $\GG^\C$. For instance, 
the connected integration of $\hh$ will be $\mathrm{H}={(\GG^\C)}^\sigma$. 

The commuting involutions produce a common direct sum decomposition of $\gg^\C$ into $\pm 1$-eigen-spaces
\begin{equation}\label{eq:commuting-eigen-spaces} 
\gg^\C=\gg\oplus i\gg=i\ss\oplus \kk\oplus \ss\oplus i\kk,
\end{equation}
where $\hh=\kk\oplus \ss$ and $\kk$ is the Lie algebra of the maximal compact subgroup 
$\KK\subset \mathrm{H}$, $\KK=\GG\cap \mathrm{H}$. 
 
Multiplication times $i$ identifies $X\subset \gg$ with a $\GG$-orbit of $i\gg\subset \gg^\C$. From now on, we shall 
work with this $\GG$-orbit, which we also denote by $X\subset i\gg$. The reason is that we want to take
advantage of the hyperbolic $\GG^\C$-orbit that contains $X$,  which we shall denote by $\cO$.  

From the Cartan decomposition $\gg^\C=\gg\oplus i\gg$ we pass to an Iwasawa decomposition of $\gg^\C$ by
choosing first $\aa\subset i\gg$ a maximal 
abelian subalgebra. Because $\hh$ is a split real form, we may take $\aa\subset i\ss\subset \hh$. The complexification 
$\aa\oplus i\aa\subset \ss\oplus i\ss$ is a Cartan subalgebra of $\gg^\C$ and, therefore, $X$ intersects $\aa$. 
In other words, the opposite 
of the restriction of a Weyl involution to $\gg$ acts on every adjoint orbit with non-empty fixed point set.
This means that for any adjoint orbit of $\gg$, a Weyl involution on $\gg$  is in the hypotheses 
of Theorem \ref{thm:canonical}.  

Next we fix a root ordering
$\Sigma=\Sigma^+\cup \Sigma^-$ and define $\nn=\sum_{\alpha\in \Sigma^+}\gg_\alpha$ so that
\[\gg^\C=\gg\oplus \aa\oplus \nn\]
is an Iwasawa decomposition of the real semisimple Lie algebra $\gg^\C$. We have exactly one point $x$ in the intersection of $X$
with the positive Weyl chamber of $\aa$:
\[X=\GG(x)\,\,\subset\,\, \cO=\GG^\C(x).\]

We start by recalling how a structure of complex homogeneous space on $X$ is obtained. Because $\GG^\C$ is a complex subgroup
the adjoint orbit $\cO$ is a complex manifold. Let $\mathrm{Z}\subset \GG^\C$ be the centralizer of $x$ and let 
$\mathrm{P}$ be the normalizer of $\mathrm{Z}$. These are complex subgroups. We have a  canonical identification $\cO\cong \GG^\C/\mathrm{Z}$ given by the action of $\GG^\C$ on $x$, and a submersion
\begin{equation}\label{eq:quasiproj-orbit} 
\GG^\C/\mathrm{Z}\cong \cO\to \GG^\C/\mathrm{P}.
 \end{equation}
The base is canonically diffeomorphic to $X$ as the action of $\GG$ on $[e]$ provides a section. This section is not holomorphic. Far from it,  
the anti-holomorphic involution $-\theta$ fixes $\ss$ and therefore fixes the centralizer $\mathrm{Z}$. Thus it descends to an 
anti-holomorphic involution on $\GG/\mathrm{Z}$ whose fixed point set is the section $\GG([e])$.

It is important to regard $X$ not just as a complex homogeneous space, but to remember the whole structure of homogeneous bundle of which
$X$ is naturally a section. Furthermore, the Iwasawa decomposition at the group level provides a linearization of 
this homogeneous bundle (\ref{eq:quasiproj-orbit}).
We state this well-known result as a lemma for further use.
\begin{lemma}\label{lem:Iwasawa} The group Iwasawa decomposition  $\GG^\C=\GG\mathrm{AN}$ induces  a ruling/affine bundle structure on 
the adjoint orbit $\cO\to \cO/\sim$ characterized by having the $\mathrm{N}$-orbit $\mathrm{N}(x)=x+\nn(x)$, $\nn(x)=\sum_{\alpha(x)>0}\gg_\alpha$, as fiber through $x$ 
  and having the action of $\GG^\C$ on $\cO$ preserving the ruling. 
  
The ruling has the following properties:
\begin{enumerate}[(i)]
\item  The identification $\GG^\C/\mathrm{Z}\cong \cO$ is a $\GG^\C$-equivariant biholomorphism from the homogeneous to the affine bundle 
\begin{equation}\label{eq:Iwasawa} \xymatrix{ 
\GG^\C/\mathrm{Z}  \ar[rr]^-{} \ar[d]_-{}  &    &  \cO \ar[d]_-{\Pi}\\
 \GG^\C/\mathrm{P} \ar[rr]^-{}  &   &            \cO/\sim         
}
\end{equation}
\item The biholomorphism intertwines the involution on $\GG^\C/\mathrm{Z}$ induced by the opposite of $\theta$ on $\GG^\C$ 
and the restriction to $\cO$ of the opposite Lie algebra involution $-\theta$. Therefore, 
it identifies the fixed point sets $\GG([e])\subset \GG^\C/\mathrm{Z}$ and
\begin{equation}\label{eq:real-form-1}
\cO^{-\theta}=\cO\cap \ss\oplus i\kk=\GG(x)=X\subset \cO.
 \end{equation}
\end{enumerate} 
\end{lemma}

We now discuss symplectic structures. The orbit $\cO$ is endowed with the (holomorphic) $\mathrm{KKS}$ symplectic form $\Omega$. We
are concerned with two submanifolds
of $\cO$ whose symplectic geometry can be read  using involutions:
One is $X\subset \cO$.  The identification of the adjoint orbit
in $\gg$ with the $\GG$-orbit $X\subset i\gg$ takes the real $\mathrm{KKS}$ symplectic form to $-\Im\Omega$. 
The other one is the orbit of $x$ by the complex subgroup $\KK^\C$, which is the fixed point set of the
holomorphic involution $\theta\sigma$. Because $\GG$ is simply connected,  the subgroup $\KK^\C$ is connected \cite[Theorem 8.1]{St}. The group $\KK^\C$
is a spherical subgroup. It acts on $X$ with a finite number of orbits \cite[Corollary 4.3]{Sp}. The open and dense orbit 
is the subset $X^*\subset X$ that appears in the statement of Theorem \ref{thm:canonical}.
 
\begin{proposition}\label{pro:fixedpoints-sections} 
The fixed point sets of the opposite involutions $-\theta$ and $-\theta\sigma$ on $\cO$ have the following properties:  
 \begin{enumerate}[(i)]
  \item The fixed point set of $-\theta$ --- which is the $\GG$-orbit $X$ --- is  Lagrangian with respect to $\Re \Omega$ and symplectic with respect to $\Im\Omega$.
  \item The fixed point set of $-\theta\sigma$ is the $\KK^\C$-orbit $\KK^\C(x)$. It is Lagrangian with respect to $\Omega$ and 
  it is a section of the Iwasawa projection
  whose image is the open dense orbit $X^*\subset X$. 
 \end{enumerate}
\end{proposition}

\begin{proof}

The opposite involution $-\theta$ on $\gg^\C$ is an anti-conjugate Lie algebra morphism. This means that
it acts  on the holomorphic Poisson manifold $(\gg^\C,\pi_\lin)$ taking orbits to orbits, but transforming the holomorphic $\mathrm{KKS}$ form into its
opposite conjugate. Because $-\theta$ restricts to $\cO$ to a biholomorphism, it sends $\Omega$ to $-\overline{\Omega}$. Thus, its fixed
point set --- which by (\ref{eq:real-form-1}) is $X$ ---  is  Lagrangian with respect to $\Re \Omega$ and symplectic with respect to $\Im\Omega$.

To prove that the fixed point set 
$\cO^{-\theta\sigma}=\cO\cap \ss\oplus i\ss$ is a section of the Iwasawa ruling we need to introduce the opposite Iwasawa ruling.
Let $\overline{\nn}$ denote the sum of the negative root spaces and let $\overline{\nn}(x)$ be the sum of those subspaces corresponding to negative 
roots not vanishing on $x$. Spreading  $x+\overline{\nn}(x)$ via the $\GG^\C$-action (or just the $\GG$-action) produces the opposite
Iwasawa ruling of $\cO$. Because $\overline{\nn}\cap \nn=\{0\}$, the fibers of both rulings intersect at most in one point. 
The restriction of $\theta\sigma$ to $\aa\subset \ss$ is minus the identity. Therefore, it exchanges $\nn(x)$ and $\overline{\nn}(x)$. Hence,
the restriction of $-\theta\sigma$ to $\cO$ exchanges the Iwasawa and the opposite Iwasawa rulings.  The consequence is that the tangent
space at a point in $\cO^{-\theta\sigma}$ must have trivial intersection with the tangent spaces of the Iwasawa and opposite Iwasawa fibers through the point. Also,
if $\cO^{-\theta\sigma}$
intersects an Iwasawa fiber in two points, then because $-\theta\sigma$ is a linear transformation
the affine line through both points would be contained in   $\cO^{-\theta\sigma}$, which contradicts the previous statement on the tangent space of 
the fixed point set. Thus, $\cO^{-\theta\sigma}$ is a section of the Iwasawa projection.

Because $\KK^\C$ is the fixed point set of the involution $\theta\sigma$ on $\GG^\C$, it follows that
$\KK^\C(x)\subset \O^{-\theta\sigma}$.
  Let us assume for the moment that $\KK^\C(x)$ projects onto $X^*$. 
 The  opposite involution $-\theta\sigma$ on $\gg^\C$ is an anti-Lie algebra morphism. 
Because it acts on $\cO$ and fixes $x$, its  fixed point set on $\cO$ is Lagrangian 
with respect to $\Omega$. Thus, its dimension equals the dimension of $X$, and, therefore, $\KK^\C(x)$ must be an open subset of it.
By \cite[Corollary 5.3]{Bi} $\KK^\C(x)$ is a closed orbit which implies that it is a connected component of the fixed point set. 
 Assume that there exists another connected component. It is a section of the Iwasawa projection 
over its image, which must necessarily have the dimension of $X$. Thus, its intersection with $X^*$ is non-empty, which contradicts
that $\cO^{-\theta\sigma}$ is a section of the Iwasawa projection. 

It remains to show that the image of the orbit $\KK^\C(x)$ by the Iwasawa projection is $X^*$. For that
it is enough to check that the dimension of  
the tangent space of the $\KK^\C$-orbit of $x\cong [e]$ in $X\cong \GG^\C/\mathrm{P}$ is the dimension of $X$. 
The  Lie algebra of $\mathrm{P}$ is
\[\mathfrak{p}=\overline{\nn}_{(x)}\oplus i\aa\oplus \aa\oplus \nn,\] 
where $\overline{\nn}_{(x)}$ denotes the sum of root spaces corresponding to negative roots vanishing on $x$.
The kernel of the differential of the Iwasawa projection at $x$ is $\nn$, which is contained in $\pp$. Thus, it suffices to show the equality:
\[[\kk\oplus i\kk,x]\oplus \pp=\gg^\C.\]
Because  $\sigma$ fixes $x$ both summands are invariant by $\sigma$. Therefore, we can equivalently show the equality of real forms
\begin{equation}\label{eq:equality-of-real-forms}
 [\kk,x]\oplus \overline{\nn}^\sigma_{(x)}\oplus \aa\oplus \nn^\sigma=\hh,
\end{equation}
which follows, for example, from \cite[sections 2 and 3]{DKV}.
\end{proof}

\section{LS submanifolds, proper K\"ahler potentials, and Euler vector fields}\label{sec:LS}

The formula for the $\mathrm{KKS}$ symplectic form at $x\in\cO$ is $\Omega_x([u,x],[v,x])=\langle x, [u,v]\rangle$,
where $\langle\cdot,\cdot \rangle$ is the (complex) Killing form on $\gg^\C$.  A similar formula holds at any point in $\cO$ due to the 
invariance of both the bracket and the Killing form by the adjoint action. Therefore, the fibers of the Iwasawa projection 
are Lagrangian submanifolds of $(\cO,\Omega)$.

A fibration with Lagrangian 
fibers has an infinitesimal action of the cotangent bundle of its base on the vertical tangent bundle (see \cite{D} for the smooth case 
and \cite[section 1]{Li} for the holomorphic one). In our case,
this infinitesimal action integrates into a fiberwise holomorphic group action   
\begin{equation}\label{eq:affine-action}
T^{*1,0}(\cO/\sim) \times_{\cO/\sim}\cO\to \cO
 \end{equation}
which turns $\cO\to \cO/\sim$ into a holomorphic affine Lagrangian bundle  with a Hamiltonian action of $\GG^\C$ \cite[Theorem 3.11]{Li}.
Affine Lagrangian bundles are obtained by gluing together pieces of cotangent bundles with their  Liouville symplectic form.
In our setting, the spherical subgroup
provides one such (large!) piece.

We denote by $\cA$ the inverse image of $X^*\subset X\cong \cO/\sim$ by the Iwasawa projection. Equivalently,
$\cA\subset \cO$ is the collection of $\KK^\C$-orbits of points in the Iwasawa fiber
$x+\nn(x)$. 

\begin{proposition}\label{pro:hol-ctg} 
There is a canonical 
extension of the inclusion of $\KK^\C(x)$ in $\cA$ to an isomorphism of affine Lagrangian bundles
\begin{equation}\label{eq:holomorphic-ctg-bundle}
 \xymatrix{ 
(T^{*1,0}\KK^\C(x),d\lambda)\ar[d]_-{} \ar[rr]^-{\chi}   &    & (\cA,\Omega) \ar[d]_-{\Pi} \\
  \KK^\C(x) \ar[rr]^-{\Pi} &   &  \cA/\sim                  
}
 \end{equation}
  defined as follows: 
 At a point $y\in \KK^\C(x)$ the Killing form and the biholomorphism $\Pi:\KK^\C(x)\to \cA/\sim$ identify a covector at 
   $y\in \KK^\C(x)$ with a unique vector 
in the tangent space to the fiber through $y$, which we add to $y$.  
\end{proposition}
\begin{proof} 
According to  \cite[Theorem 3.11]{Li}  $\Pi:(\cO,\Omega)\to \cO/\sim$ is an affine Lagrangian bundle. Let us assume 
that the affine structure on fibers and the action of $T^{*1,0}\cO$ on the vertical bundle
is the natural one coming from the embedding $\cO\subset \gg^\C$, by which we mean the following:
 The affine structure is given by the Iwasawa ruling 
described in Lemma \ref{lem:Iwasawa}.
The action of a  cotangent vector at a point on $\cO/\sim$ is by addition of the vector tangent to its Iwasawa fiber with which it is in duality
with respect to the Killing  form. This duality at $x$ uses that the complex vector space $\nn(x)$ is isotropic, the identification of  
the cotangent fiber of $T^{1,0}\cO/\sim$ at $x$ with
$\{v+\theta v,|\,v\in \nn(x)\}$, and the duality (over the reals) between 
the latter space and $\nn(x)$ established by the real part of the Killing form --- which can be deduced from the non-degeneracy of 
$\Re\langle \cdot,\theta\cdot\rangle$. 

 By item (ii) in Proposition \ref{pro:fixedpoints-sections}, $\KK^\C(x)$ is a Lagrangian section.
Therefore, $\chi$ as described in the statement 
is exactly the action map (\ref{eq:affine-action}) applied to the Lagrangian section $\KK^\C(x)$ of the affine Lagrangian bundle $(\cO,\Omega)$. Hence, 
it defines an isomorphism of affine Lagrangian bundles \cite[Proposition 1.6]{Li}.

The affine bundle structure on $\Pi:(\cO,\Omega)\to \cO/\sim$ and action of $T^{*1,0}\cO$   
is described in \cite{Li} as follows. First, one
builds an affine Lagrangian bundle $(E_x,\Omega_x)$ by symplectic induction \cite[Section 2.4]{Li}. 
This means Hamiltonian reduction of $(T^*\GG^\C,d\lambda)$ at $x$ with respect
to the action of a standard parabolic subgroup $\mathrm{P}$. The affine Lagrangian bundle structure on $(T^*\GG^{\C},d\lambda)$ is the natural one:
A covector at a point in $\GG^\C$ acts on the fiber by addition. This affine Lagrangian bundle structure descends to the quotient $(E_x,\Omega_x)$. 
Second, the momentum map $E_x\to {\gg^\C}^*$ restricts to an affine map on fibers and takes the action of $T^{*1,0}\cO$ on fibers to addition of 
covectors \cite[section 2.5]{Li}. As we transfer the structure from cotangent to tangent bundle,
we have to replace addition of a covector on the base by 
addition of the only vector tangent to the affine fiber furnished by the Killing form. Third, if we choose 
$\mathrm{P}$ to be the parabolic subgroup  with Lie algebra $\overline{\nn}_{(x)}\oplus i\aa\oplus \aa\oplus \nn$,
 then the momentum
map is a symplectomorphism $(E_x,\Omega_x)\to (\cO,\Omega)\subset \gg$ which takes affine fibers to Iwasawa fibers \cite[Theorem 3.11]{Li}. Therefore
the affine Lagrangian bundle structure on $(\cO,\Omega)$ is the natural one coming from the embedding in $\gg^\C$.
\end{proof}

In order to ease the notation, 
from now on we will use the identification of $(\cA,\Omega)\to \cA/\sim$ with the cotangent bundle of $\KK^\C(x)$ without writing $\chi$.
Thus, we may now regard $X^*$ as a section of the cotangent bundle that is Lagrangian and symplectic with respect to the real and imaginary parts of
the Liouville symplectic form, respectively. This is an $\mathrm{LS}$ submanifold in the terminology of \cite[Section 2]{Do}. Such submanifolds were introduced 
as geometric counterparts of K\"ahler potentials.  

\begin{proposition}\label{pro:LS} The symplectic form $-\Im\Omega$ on $X^*$ has a K\"ahler potential
\begin{equation}\label{eq:Kahler-potential}
-\Im\Omega=i\partial\bar{\partial}h
\end{equation}
normalized by the condition that it vanishes on $x\in X^*$. 
\end{proposition}
\begin{proof}
The submanifold $X^*\subset T^{*1,0}\KK^\C(x)$ is the graph of a section $\varsigma$. By item (i) in Proposition 
\ref{pro:fixedpoints-sections}, it is an $\mathrm{LS}$ section. It follows from \cite[Section 2]{Do} that:
\begin{itemize}
 \item the pullback $\varsigma^*\Im d\lambda$ is a (real) symplectic form on $\KK^\C(x)$;
 \item under the identification  $T^{*1,0}\KK^\C(x)\cong T^*\KK^\C(x)$ which takes a complex linear form to its real part,
 the section $\varsigma$ is a closed 1-form $\beta$;
 \item if $\beta$ is exact,  then any primitive $\eta\in C^\infty(\KK^\C(x))$ satisfies:
$2i\bar{\partial}\partial \eta=\varsigma^*\Im d\lambda$. 
\end{itemize}
In such a situation,  pulling back again from $\KK^\C(x)$ to $X^*$ we would conclude that 
$h={\varsigma^{-1}}^*\frac{1}{2}\eta\in  C^\infty(X^*)$ is a K\"ahler potential for $-\Im\Omega$.

Therefore, it remains to discuss the exactness of $\beta$. Because $\KK^\C$ is invariant under 
the Cartan involution $\theta$,  the Cartan decomposition of $\GG^\C$ induces a Cartan decomposition 
$\KK^\C=\KK\exp(i\kk)$.
Thus, if we let $\KK_x$ denote the isotropy group of the action of $\KK$ on $x$, we obtain diffeomorphism $\KK^\C(x)\cong T^*\KK/\KK_x$,
where the orbit $\KK(x)$ goes to the zero section. Hence, the compact $\KK$-orbit $\KK(x)$ is a deformation retract of $\KK^\C(x)$.
The submanifolds $X^*$ and $\KK^\C(x)$ intersect precisely in that $\KK$-orbit $\KK(x)$. Therefore, the closed 1-form 
$\beta \in \Omega^1(\KK^\C(x))$
vanishes along $\KK(x)$, which implies that it is exact. 
\end{proof}

\begin{proposition}\label{pro:potential-properties} The  K\"ahler potential $h$ in (\ref{eq:Kahler-potential})  has the following properties:
 \begin{enumerate}[(i)]
 \item  It is real analytic.
  \item It is invariant under the commuting actions of $\KK$ and the involution $\sigma$.
  \item It is Morse--Bott, positive, and vanishes exactly on $\KK(x)$.
  \item It is proper. 
 \end{enumerate}
\end{proposition}
\begin{proof} Because $X^*$ is the fixed point set of an anti-holomorphic involution, it is a real analytic submanifold. Therefore, the
 section $\varsigma\in \Omega^{*1,0}(\KK^\C(x))$ and its corresponding closed 1-form $\beta\in \Omega^1(\KK^\C(x))$ are real analytic. Hence
 so are any of the primitives of $\beta$.
 
 The action of $\sigma$ on $\GG^\C$ defines a semidirect product group $\Z_2\ltimes \GG^\C$. Because $\sigma$ fixes the 
 centralizer $Z$ and its normalizer $\mathrm{P}$, the group acts by bundle transformations on $\GG^\C/\mathrm{Z}\to \GG^\C/\mathrm{P}$. This action 
 is transferred by the biholomorphism (\ref{eq:Iwasawa}) to an action on $\cO\to \cO/\sim$ by affine bundle biholomorphisms. 
 Both $\cA$ and the orbit $\KK^\C(x)$ are invariant by the subgroup $\Z_2\ltimes \KK^\C$. The cotangent  lift of the latter action
 is identified
 by $\chi$ with the action on $\cA\to \cA/\sim$. The section $X^*$ is invariant by the subgroup  $\Z_2\times \KK$. 
 Therefore, the closed 1-form $\beta$ is  $\Z_2\times \KK$-invariant and so  
  any of its primitives are. The involution $\sigma$ preserves the Iwasawa fibers and,  therefore, projecting along fibers identifies $\KK^\C(x)$ and $X^*$ in a 
  $\Z_2\times \KK$-equivariant
  fashion. Thus, the pullback of $\beta$ to $X^*$ is also a $\Z_2\times \KK$-invariant exact 1-form and therefore the K\"ahler potential $h$
 is invariant by the commuting actions of $\sigma$ and $\KK$.

By definition the critical points of $h$ are the intersection of the graph of $\varsigma$ with the zero section:
$X^*\cap \KK^\C(x)=\KK(x)$.
The Morse--Bott condition requires first transversality of the previous intersection. Such intersection 
is the real form for the involution $\sigma$ on both $\KK^\C(x)$ and $X^*$ (the latter with respect to the quotient complex structure).
Because the Iwasawa projection intertwines both anti-holomorphic involutions, it follows that 
\[T_x\KK^\C(x)\cap T_xX^*=T_x\KK(x).\]
Second, by $\KK$-invariance we just need to check the Hessian condition on the normal bundle to $\KK(x)$ at $x$: $T_xX^*/T_x\KK$.
We shall take as representative of the quotient tangent space 
\[\Pi_*T_x\exp(i\kk)(x)=\Pi_*[i\kk,x].\]
Let $u\in \kk$ be decomposed as $u=\sum_{\alpha\in \Sigma^+} u_\alpha+\theta u_\alpha$, $u_\alpha\in \gg_\alpha^\sigma $.
We have:
\[[iu,x]=i\sum_{\alpha\in \Sigma^+} [u_\alpha+\theta u_\alpha,x]=-i\sum_{\alpha\in \Sigma^+} \alpha(x)(u_\alpha-\theta u_\alpha)=-\sum_{\alpha\in \Sigma^+} \alpha(x)i(u_\alpha-\theta u_\alpha),  \]
and therefore 
\begin{equation}\label{eq:proj-tg}
\Pi_*([iu,x])=i[u,x]+\sum_{\alpha\in \Sigma^+} 2i\alpha(x)u_\alpha=\sum_{\alpha\in \Sigma^+}\alpha(x) i(u_\alpha+\theta u_\alpha).
\end{equation}
If we let $u'=-\sum_\alpha i(u_\alpha-\theta u_\alpha)\in \gg$,  then we can write
\[[u',x]=\sum_\alpha i\alpha(x)(u_\alpha+\theta u_\alpha)=\Pi_*([iu,x]).\]
Therefore, the intrinsic derivative of ${\varsigma^{-1}}^*\beta $ in the direction of $[u',x]$ reads:
\[\nabla {\varsigma^{-1}}^*\beta ([u',x])=\sum_{\alpha\in \Sigma^+} 2i\alpha(x)u_\alpha.\]
To go from the intrinsic derivative of the 1-form to the Hessian of the primitive, we 
use the linear isomorphism from $\nn^\sigma$ to $T^*_xX^*$ given by twice the real part of the Killing form:
\[\begin{split}\mathrm{Hess}_x([u',x],[u',x])& =2\Re\langle \nabla{\varsigma^{-1}}^*\beta ([u',x]),u' \rangle=
2\Re \langle  \sum_{\alpha\in \Sigma^+} 2i\alpha(x)u_\alpha, -\sum_{\alpha\in \Sigma^+} i(u_\alpha-\theta u_\alpha)\rangle= \\
&=\sum_{\alpha\in \Sigma^+} 4\alpha(x)\Re\langle u_\alpha,-\theta u_\alpha\rangle >0,
\end{split}
\]
where in the last equality, we used that $\nn$ is isotropic for the Killing form and $2\Re\langle \cdot,-\theta\cdot\rangle$ is an inner product on
$\hh$ for which root spaces are mutually orthogonal, and in the final inequality, we used  that $x$ belongs to the positive Weyl chamber. 

From the strict positivity of the Hessian we conclude that $\KK(x)$ are minima, and thus $h$ is a positive function. 

The $\KK$-equivariant diffeomorphism $\Pi^{-1}:X^*\to \KK^\C(x)$ allows us 
to regard $h$ as a $\KK$-invariant function on a complex homogeneous space. This function is strictly plurisubharmonic as it is a K\"ahler potential.
 Because $h$ attains a minimum  by \cite[Theorem 1]{AL},
it must be proper. 
\end{proof}

We define $Y\in \mathfrak{X^*}$ to be one half of the gradient vector field of $h$. This 
is the vector field in the statement of Theorem \ref{thm:canonical}.

\begin{proof}[Proof of Theorem \ref{thm:canonical}]

By (\ref{eq:Kahler-potential}) $Y$ is a Liouville vector field. By items (iv) and (iii) in Proposition \ref{pro:potential-properties},
its trajectories for negative time have limit on $X^*$ and  its linearization
 $\nabla Y: TX^*|_{\KK(x)}\to TX^*|_{\KK(x)}$ has rank
half the fiber dimension of the tangent bundle. In particular, the dynamics of $Y$ are normally hyperbolic on $X^*$, and thus by the unstable manifold
theorem trajectories
for negative time have a unique limit point in $X^*$ \cite{HPS}. If $Y$ were complete, then by  Nagano's characterization of cotangent bundles 
\cite[Theorem 4.1]{Na},
the vector field would give rise to a symplectomorphism from $X^*$ to $T^*\KK(x)$ taking $Y$ to the Euler vector field. 
The vector field $Y$ is backwards complete. The construction in \cite[Theorem 3]{Na} 
also extends for vector fields that are backwards complete, but 
the outcome is a symplectomorphism $\psi$ onto a domain $D\subset T^*\KK(x)$, which is (fiberwise) star-shaped.

The symplectic form $-\Im\Omega$ and K\"ahler metric are $\KK$-invariant, and, therefore, $Y$ is $\KK$-invariant as well. This
symmetry extends to $\Z_2\times \KK$-invariance by  item (ii) in Proposition \ref{pro:potential-properties}. The 
cotangent lift of the left action of $\KK$ on $\KK(x)$ is by vector bundle automorphisms. Therefore,
the Euler vector field is $\KK$-invariant. This symmetry extends as well to  $\Z_2\times \KK$-invariance, where the $\Z_2$-action
comes from the involution $\iota$ that sends a covector to its opposite. 
Because Nagano's symplectomorphism $\psi$ is characterized by taking integral curves of $Y$ to integral curves of the Euler vector field, and
because the action of $\Z_2\times \KK$ preserves $X^*$, the symplectomorphisms 
must be $\Z_2\times \KK$-equivariant. A posteriori, we deduce that the fiber of the unstable normal bundle of $Y$ at a point in $X^*$ is the 
-1-eigen-space of $\sigma$.

Both $(X^*,-\Im\Omega)$ and $(T^*\KK(x),d\lambda)$ are $\KK$-Hamiltonian spaces. Thus, $\psi$ pulls back 
any momentum map for the latter space to a momentum map to the former space. The canonical momentum map $\mu$ for the cotangent lift 
of the action sends the zero section to zero. The natural momentum map $\nu$ for  $(X^*,-\Im\Omega)$ is the
restriction of the projection $\pp\oplus i\kk\to \kk$ followed by multiplication times $-i$ and the isomorphism given 
by twice the real part of the Killing form. Because $\KK(x)$ is mapped to zero, $\psi$ must
intertwine both momentum maps. 

The canonical momentum map $\mu:T^*\KK(x)\to \kk^*$ is linear on fibers. Therefore, it relates Euler vector fields. Because 
$\nu=\mu\circ \psi$ and $\psi$ relates $Y$ to the Euler vector field, then $\nu$ also relates $Y$ to the Euler vector field of $\kk^*$.

A subset of $T^*\KK(x)$ is bounded if and only if its image by the canonical (proper) momentum map $\mu$ is bounded. By the previous
paragraph $\mu(D)=\nu(X^*)$. Because $\nu(X^*)\subset \nu(X)$ the former subset is bounded. Hence, $D\subset T^*K$ is a bounded subset.   
\end{proof}

\begin{example}\label{ex:su(2)} We let $\GG=\SU(2)$ and we fix $x\in i\gg\subset \sl(2,\C)$
\[x=\begin{pmatrix}
     1 & 0 \\ 0 & -1
    \end{pmatrix}.
\]
The holomorphic adjoint orbit in which we are interested  is $\cO=\SL(2,\C)(x)$.
The Lie algebra $\nn$ of positive root spaces is the subspace of upper triangular
matrices with zeros in the diagonal. The Iwasawa fiber over $x$ described in Lemma \ref{lem:Iwasawa} is 
\[\mathrm{N}(x)=\begin{pmatrix}
                           1 &  \zeta\\
                           0 & 1
                          \end{pmatrix}\cdot x=\left\{\begin{pmatrix}
                           1 &  -2\zeta\\
                           0 & -1
                          \end{pmatrix}\,|\, \zeta\in \C\right\}.  
\]
The complex involution $\theta\sigma$ on $\sl(2,\C)$ takes a matrix to its transpose. 
Therefore, the spherical group is 
$\SO(2,\C)$. The exponential map is given by 
\begin{equation}\label{eq:exponential}z\mapsto\begin{pmatrix} \cos(z) & -\sin(z)\\
                           \sin(z) & \cos(z)
                          \end{pmatrix},\quad z=a+ib\in \C.
                          \end{equation}                   
The orbit $\SO(2,\C)(x)$ is the holomorphic Lagrangian of $(\cO,\Omega)$ given by 
\[\left\{\begin{pmatrix}
                           z & w\\
                           w & -z
                          \end{pmatrix}\,|\, z,w\in\C,\,\, z^2+w^2=1\right\}. 
\]
Next we compute the complex 1-form $\varsigma$ on $\SO(2,\C)(x)$ defined by the section $X=\SU(2)(x)$ as described in Proposition \ref{pro:LS}.
We use that 
if $h\in \SO(2,\C)$ has Iwasawa decomposition $h=gle$ then $\varsigma$ is the dual of
$g\cdot x-h\cdot x\in g\cdot \nn$.
Therefore, for $c\in \kk$ we have 
\begin{equation}\label{eq:sl(2)1-form}
 \varsigma_{h\cdot x}(h_*c)=\langle g\cdot x-h\cdot x,h\cdot c\rangle=\langle x-le\cdot x,le\cdot c\rangle=4\mathrm{tr}(x-l\cdot x)(le\cdot c).
\end{equation}
The Gram--Schmidt algorithm for the standard Hermitian inner product in $\C^2$ returns the Iwasawa factorization:
\begin{equation}\label{eq:GS}\begin{pmatrix} \cos(z) & -\sin(z)\\
                           \sin(z) & \cos(z)
                          \end{pmatrix}=\begin{pmatrix}\frac{\cos(z)}{d} & -\frac{\overline{\sin(z)}}{d}\\
                           \frac{\sin(z)}{d} & \frac{\overline{\cos(z)}}{d}
                          \end{pmatrix}\begin{pmatrix} 1 & \cos(z)\overline{\sin(z)}-\sin(z)\overline{\cos(z)}\\
                           0 & 1
                          \end{pmatrix}\begin{pmatrix} d & 0\\
                           0 & \frac{1}{d}
                          \end{pmatrix},
\end{equation}                          
where $d^2=\cos(z)\overline{\cos(z)}+\sin(z)\overline{\sin(z)}$. If we pull back $\varsigma$ to $\C$ via the exponential map (\ref{eq:exponential})
so that $z=a+ib$ and $c\in \C$, then (\ref{eq:sl(2)1-form}) yields
\[\varsigma_z(c)=-8ic\frac{\sinh(2b)}{\cosh(2b)}.\]
The primitive of 
\[\frac{1}{2}\Re\varsigma=4\frac{\sinh(2b)}{\cosh(2b)}db\]
is the (normalized) K\"ahler potential 
\[h(a,b)=2\ln \cosh(2b).\]
Therefore, the pullback to $\C$ of the KKS symplectic form is 
\[\w=\frac{1}{2}d\circ J^*\left(\frac{1}{2}\Re\varsigma\right)=\frac{4}{\cosh^2(2b)}da\wedge db.\]
The open dense orbit $X^*\subset \SU(2)(x)$ can be explicitly parameterized
by using the the first factor of the Iwasawa decomposition (\ref{eq:GS}):
\begin{equation}\label{eq:hyp-coord}\begin{pmatrix} \cos(z) & -\sin(z)\\
                           \sin(z) & \cos(z)
                          \end{pmatrix}\cdot x=\frac{1}{\cosh(2b)}\left(\cos(2a)\begin{pmatrix} 1 & 0\\ 0 & -1\end{pmatrix}+
                          \sin(2a) \begin{pmatrix} 0 & 1\\ 1 & 0\end{pmatrix}\right)+
                         \frac{\sinh(2b)}{\cosh(2b)}\begin{pmatrix} 0 & -i\\ i & 0\end{pmatrix}.
\end{equation}
This implies that the cylindrical coordinates are
 \[\theta=2a, \quad z=\frac{\sinh(2b)}{\cosh(2b)}.\] 
Thus the pullback to $\C$ of  $d\theta \wedge dz$ is
\[2da\wedge 2\frac{1}{\cosh^2(2b)}db=\frac{4}{\cosh^2(2b)}da\wedge db,\]
which as expected is $\w$.
The Hamiltonian vector field of $h$ is 
\[X_h=\sinh(2b)\cosh(2b)\frac{\partial}{\partial a}=\frac{1}{2}\sinh(4b)\frac{\partial}{\partial a},\]
and therefore the gradient vector field is
\[Y=\frac{1}{2}\sinh(2b)\cosh(2b)\frac{\partial}{\partial b}.\]
Its push forward by the coordinate chart is a vector field tangent to the meridian through $x$. Its height component is:
\[\frac{1}{2}\sinh(2b)\cosh(2b)\frac{\partial}{\partial b}\left(\frac{\sinh(2b)}{\cosh(2b)}\right) \frac{\partial}{\partial z}=
\frac{1}{2}\sinh(2b)\cosh(2b)\left(\frac{2}{\cosh^2(2b)}\right) \frac{\partial}{\partial z}=z\frac{\partial}{\partial z}.\]
Hence, in cylindrical coordinates, $Y$ is the Euler vector field of $T^*\mathbb{RP}^1$.

\end{example}

\subsection{Arbitrary involutions}\label{ssec:involutions}

Let us suppose that $\sigma$ is an anti-complex Lie algebra involution on $\gg^\C$ that commutes with the Cartan involution $\theta$. To place ourselves
in the hypotheses of Theorem \ref{thm:canonical}, we must require that $X$ intersects the fixed point set $\ss$ of the restriction of $\sigma$ to $i\gg$. We need to make several adjustments
in the constructions of the previous sections. 

\begin{enumerate}
 \item The maximal abelian Lie algebra of $i\gg$ is now chosen to be $\sigma$-stable (standard). This means it splits as 
 $\aa\oplus i\tt$, $\aa\subset \ss$, $i\tt\subset i\kk$.
  Upon restriction of the roots of $(\gg^\C,\aa\oplus i\tt)$ to $\aa$ we get a (reduced) root system for $(\hh,\aa)$.
 We first choose an ordering of the reduced root system and then extend it to an ordering of the root system: $\Sigma=\Sigma^+\cup \Sigma^-$. This means that 
 for a non-imaginary root --- a root non-vanishing on $\aa$ --- its positivity is determined by the positivity of its restriction to $\aa$.
 \item The intersection of $X$ with $\ss$ is non-empty. Because all maximal abelian subalgebras of $\ss$ are conjugate under the action of $\KK$
 we may assume that $\aa$ intersects $X$. In that intersection, there is a unique point $x$ lying in the positive Weyl chamber 
 of the ordered reduced root system/ordered root system.
 \item Lemma \ref{lem:Iwasawa} concerns the Cartan involution and thus holds true.
 \item Item (ii) in Proposition \ref{pro:fixedpoints-sections} is about properties of the spherical subgroup $\KK^\C$, which is the fixed point set of the holomorphic involution $\theta\sigma$.
 We need  further analysis of the affine fibers $x+\nn(x)$ and $x+\overline{\nn}(x)$  to deduce that their intersection and their behavior under the action of $-\theta\sigma$ is the same as
 the one described when $\sigma$ is the Weyl involution: First, we recall that every imaginary root vanishes on $\aa$ and therefore its corresponding root space acts trivially on $x\in X\cap \ss$. If we 
 let $\nn_{(\ss)}$ and $\nn(\ss)$ denote the eigen-spaces for the positive imaginary roots and the eigen-space for the positive non-imaginary roots, respectively,
 we obtain a direct sum decomposition of subalgebras
 \[\nn=\nn(\ss)\oplus \nn_{(\ss)}.\]
 Therefore, we have the corresponding factorization  
 \[\mathrm{N}=\mathrm{N}(\ss)\mathrm{N}_{(\ss)}.\]
 Because $\mathrm{N}_{(\ss)}$ is contained in $\mathrm{Z}$, the centralizer of $x$, we may assume in the proof of Proposition \ref{pro:fixedpoints-sections}  that the fiber $x+\nn(x)$
 of the Iwasawa  ruling of $\cO\to X$ is an affine subspace of
 $x+\nn(\ss)$. Likewise, in the opposite Iwasawa decomposition, we may assume that the fiber over $x$
 is an affine subspace of $x+\overline{\nn}(\ss)$. Because $\nn(\ss)\cap \overline{\nn}(\ss)=\{0\}$
 and $\sigma$ fixes the subspaces $\nn(x)\subset \nn(\ss),\overline{\nn}(x)\subset \overline{\nn}(\ss)$  (for non-imaginary roots their positivity or negativity is dictated by their 
 restriction to $\aa\subset \ss$) and the direct sum $\overline{\nn}_{(\ss)}\oplus \nn_{(\ss)}$, Proposition  \ref{pro:fixedpoints-sections} remains valid. 
 \item Propositions \ref{pro:hol-ctg} and \ref{pro:LS} in Section \ref{sec:LS} remain valid, as they hinge on Proposition \ref{pro:fixedpoints-sections}
 and general properties of holomorphic cotangent bundles and their sections. The proof of Proposition \ref{pro:potential-properties}
 uses two ingredients involving the involution $\sigma$: first,  that the Iwasawa projection intertwines the involution $\sigma$ on the total space $\cO$ and on the base $X$. This holds
 true for arbitrary involutions because as noted we can consider as Iwasawa fiber the $\sigma$-invariant subalgebra $\nn(\ss)$. Second,
 a decomposition of a vector $u\in \kk$ parameterized by positive roots. For a general involution we should replace the positive roots by the positive non-imaginary
 roots:
 \[u=\sum_{\alpha\in \Sigma^+(\ss)} u_\alpha+\theta u_\alpha,\quad u_\alpha\in \gg_\alpha^\sigma.\]
 Note that this is equivalent to using the positive roots of the reduced root system and a basis of the corresponding eigen-spaces (which may have multiplicity
 greater than one). 
\end{enumerate}
With the previous adjustments Theorem \ref{thm:canonical} is valid for arbitrary Lie algebra involutions on $\gg$.

\begin{example}\label{ex:flags} Any regular orbit of $\mathfrak{su}(3)$ is diffeomorphic to the 
manifold of full flags in $\C^3$. If we apply Theorem \ref{thm:canonical} to any regular orbit and the Weyl involution that conjugates
the coefficients of a matrix, then we obtain a symplectomorphism to a domain of the cotangent bundle over the manifold of real full flags --- a manifold that
is diffeomorphic to the quotient of $\SU(2)$ by the quaternions (see e.g. \cite[p. 335]{DKV}). 

In this example, we shall apply Theorem  \ref{thm:canonical} to  the involution
\[\tau: \sl(3,\C)\to \sl(3,\C),\quad z\mapsto -I_{2,1}z^*I_{2,1},\quad I_{2,1}=\begin{pmatrix}
           1 & 0 & 0 \\
           0 & 1 & 0 \\
           0 & 0& -1
          \end{pmatrix}\]
and to the (regular) orbit of  
\[x=\begin{pmatrix}
           0 & 0 & i \\
           0 & 0 & 0 \\
           -i & 0& 0
          \end{pmatrix}.\]
The involution $\tau$ fixes $x$ and          
the spherical subgroup $\KK^\C$ associated to $\tau$ is
\begin{equation}\label{eq:spherical-flags}
\begin{pmatrix}
           z & w & 0 \\
           \mu & \nu & 0 \\
           0 & 0& \epsilon
          \end{pmatrix},\quad  z,w,\mu,\nu,\epsilon\in \C,\quad (z\nu-w\mu)\epsilon=1.
\end{equation}          
          
Our main objective is to describe the relation between the orbits of the action of $\KK^\C$ in $X=\SU(3)(x)$ and the Gelfand--Zeitlin map
\[\lambda: X\to \R^3,\quad y \mapsto (\lambda^{(1)}(y),\lambda^{(2)}_1(y),\lambda^{(2)}_2(y)),\]
where $\lambda^{(i)}_j$ is the jth eigen-value (in decreasing order) of the ith principal minor. 
              \begin{figure}\label{fig:1}
\centering
\includegraphics[height=5.5cm]{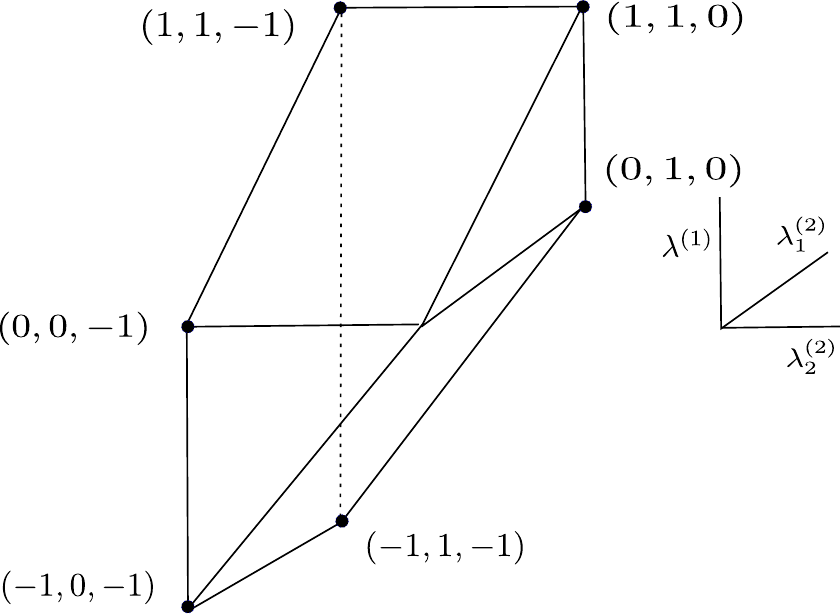}
\caption{The image of the Gelfand-Zeitlin map}
\end{figure}

The image of the Gelfand--Zeitlin map is the polytope showed in Figure \ref{fig:1} (see \cite[section 2.3]{NU} for a detailed account 
of the Gelfand-Zeitlin system).
By definition $x$
is mapped to the origin. The remaining six vertices are the image of the diagonal matrices in $X$. For example, the one mapped 
to $(1,1,0)$ is
\begin{equation}\label{eq:conj3}y=s\cdot x=\begin{pmatrix} 1 & 0 & 0 \\
          0 & 0 & 0 \\
          0  & 0 & -1
          \end{pmatrix},\quad s=\frac{\sqrt{2}}{2}\begin{pmatrix} 1 & 0 & i \\
          0 & 1 & 0 \\
          i  & 0& 1
          \end{pmatrix}.
\end{equation}
There are six  $\KK^\C$-orbits which are the preimages of the three vertical edges --- the closed orbits --
the preimages of the interior of the two faces that do not contain the origin, and 
the preimage of the complement of the two faces
--- the dense open orbit $X^*$. For instance,  $w\in X$ belongs to the preimage of the segment joining the vertices $(1,1,-1)$ and $(-1,1,1)$ if
and only if the eigen-lines of the eigen-values $\pm 1$ are contained in the plane $\C^2\subset \C^3$ spanned by the first two vectors of the canonical basis.
To describe  $\KK^\C(w)$ we 
write $w=g\cdot y$, $g\in \SU(3)$, and we let the $\SU(3)$-factor in the Iwasawa/Gram-Schmidt factorization of $\KK^\C g$ act on $w$. Because matrices in $\KK^\C$
preserve $\C^2$,  matrices in $\KK^\C g$ take the plane spanned by the first and third vectors of the canonical basis to the plane $\C^2$ and their second 
and third columns are orthogonal. This implies that the $\SU(3)$-factor coming from the Gram-Schmidt factorization of matrices in $\KK^\C g$
also  takes the plane spanned by 
the first and third vectors of the canonical basis to the plane $\C^2$. Therefore, the orbit $\KK^\C(w)$ is contained in the preimage of the segment.
The previous argument also shows that the preimage of the segment is the $\KK$-orbit of any of its points.
Hence it is also the orbit of the complexified
spherical subgroup $\KK^\C$. The description of the two open orbits of complex dimension 2 and of the open dense orbit is analogous.

The second objective is a  partial description  of the K\"ahler potential on $X^*=\KK^\C(x)$ produced by Theorem \ref{thm:canonical}. 

The isotropy subgroup at $x$ for the action of $\KK^\C$  is the one-dimensional torus
\[\begin{pmatrix}
           \lambda & 0 & 0 \\
           0 & \lambda^{-2} & 0 \\
           0 & 0& \lambda
          \end{pmatrix},\quad  \lambda\in \C^*.\]
The subgroup $\SL(2,\C)\subset \KK^\C$ obtained by setting $\epsilon=1$ in (\ref{eq:spherical-flags})
is a full slice to the right action of the above one-dimensional torus on $\KK^\C$ and intersects each orbit in one point. 
Therefore, the orbit of $x$ by $\KK^\C$ (both in $X=\SU(3)(x)$ and in $\cO=\SL(3,\C)(x)$) is the same as the free orbit by $\SL(2,\C)$. In particular
we see that the Lagrangian $X^\tau=\KK(x)$ is diffeomorphic to the free orbit of the maximal compact subgroup $\SU(2)\subset \SL(2,\C)$, that is, it is a
(well-known)
 Lagrangian 3-sphere.

To construct the complex 1-form $\varsigma$ on $\SL(2,\C)$ as in (\ref{eq:sl(2)1-form}) it is necessary to obtain the Iwasawa decomposition for $h\in \SL(2,\C)$.
The maximal torus invariant by $\tau$ where $x$ lies is not inside the diagonal matrices. Hence, we need to conjugate it to a diagonal
torus, apply the Gram-Schmidt algorithm, and then conjugate back the factorization (cf. \cite{Sa}). 
More precisely, if for any $h\in \SL(2,\C)$ we write $hs^*=gle$, where 
$s$ is defined in (\ref{eq:conj3}), then the expression for $\varsigma$ in 
(\ref{eq:sl(2)1-form}) becomes
\begin{equation}\label{eq:fl(3)}
 \zeta_{h\cdot x}(h_*c)=\langle x-s^*les\cdot x,s^*les\cdot c\rangle=6\mathrm{tr}(x-s^*les\cdot x)(s^*les\cdot c),\quad c\in \sl(2,\C).
\end{equation}
The explicit formula for $\varsigma$ is rather involved due to lack of compatibility of matrices 
in $\SL(2,\C)s^*$ with the Gram--Schmidt algorithm.
It becomes  tractable if we confine ourselves to the curve  
\[\R\subset \sl(2,\C)\overset{\mathrm{exp}}{\longrightarrow}\SL(2,\C),\quad b\mapsto \begin{pmatrix}
           b & 0 & 0 \\
           0 & -b & 0 \\
           0 & 0& 0
          \end{pmatrix}\overset{\mathrm{exp}}{\longrightarrow}\begin{pmatrix}
           e^b & 0 & 0 \\
           0 & e^{-b} & 0 \\
           0 & 0& 1
          \end{pmatrix}.
\]
Right multiplication by $s^*$ of an element in the curve has the following factorization:

\[\frac{\sqrt{2}}{2}\begin{pmatrix} e^b & 0 & -ie^b \\
          0 & e^{-b} & 0 \\
          -i  & 0 & 1
          \end{pmatrix}=\begin{pmatrix}
           \frac{e^b}{\sqrt{e^{2b}+1}} & 0 & \frac{-i}{\sqrt{e^{2b}+1}} \\
           0 & 1 & 0 \\
           \frac{-i}{\sqrt{e^{2b}+1}} & 0 & \frac{e^b}{\sqrt{e^{2b}+1}}
          \end{pmatrix}\begin{pmatrix} 1 & 0 & \frac{i(1-e^{2b})}{2e^{b}} \\
          0 & 1 & 0 \\
          0  & 0& 1
          \end{pmatrix}\frac{\sqrt{2}}{2}\begin{pmatrix}
          \sqrt{e^{2b}+1} & 0 & 0 \\
           0 &  e^{-b} & 0 \\
           0 & 0 & \frac{2e^b}{\sqrt{e^{2b}+1}}
          \end{pmatrix}
          \]
The action of an element of the curve $h$ on $x$ is obtained by letting the special unitary factor above act on $sx$:
\[\frac{1}{2}\begin{pmatrix}
          \frac{e^{2b}-1}{e^{2b}+1} & 0 & \frac{2ie^{b}}{e^{2b}+1} \\
           0 & 0 & 0\\
          \frac{-2ie^{b}}{e^{2b}+1} & 0 &  \frac{1-e^{2b}}{e^{2b}+1}
          \end{pmatrix}
\] 
Its image by the Gelfand--Zeitlin map is the segment from the origin to $(1,0,0)$ (open in that edge).
The restriction of $\varsigma$ in (\ref{eq:fl(3)}) to points $b\in \R\subset \sl(2,\C)$ yields:
\[\varsigma_{b}(c)=3c_1e^{3b}\frac{e^{2b}-1}{e^{2b}+1},\quad c=\begin{pmatrix} c_1 & c_2 & 0\\
                                                          c_3 & -c_1 & 0\\ 0 & 0 & 0
                                                         \end{pmatrix}\]
Therefore, the pullback of the normalized K\"ahler potential is
\[h(b)=e^{3b}-\frac{2}{3}e^b+\frac{2}{3}\arctan{e^b}-\frac{1}{3}-\frac{\pi}{6}.\]

\end{example}

\section{The flow for imaginary time}\label{sec:wick-rotation}

The commuting involutions $\theta$ and $\sigma$ produce a common direct sum decomposition of $\gg^\C$ in $\pm 1$-eigen-spaces
\[\gg^\C=\gg\oplus i\gg=i\ss\oplus \kk\oplus \ss\oplus i\kk.\]
The product of the trivial vector field on $\ss^\C$ and minus the holomorphic Euler vector field on $\kk^\C$ is a vector field on $\gg^\C$ 
whose flow for time $i\tfrac{\pi}{2}$ intertwines $-\theta$ and $\sigma$. This flow is only compatible with the complex linear structure on $\gg^\C$,
and not with the Lie brackets (cf. Remark \ref{rem:hol-disk}).

The orbit $X=\GG(x)$ sits inside $\ss\oplus i\kk$. The adjoint orbit $X^\vee$ in the statement of Theorem \ref{thm:duality}
is $\mathrm{H}(x)\subset \ss\oplus \kk$, $\hh=\kk\oplus \ss$. The involution $\sigma$ fixes $x$ and $X^\vee$ is the corresponding real form of $\cO=\GG^\C(x)$:
$\mathrm{H}(x)=\cO^\sigma$.
The involutions $-\theta$ and $\sigma$  act on $X^\vee$ and $X$, respectively, with common fixed point set $\KK(x)$.
As for symplectic structures, the KKS symplectic forms on the 
real adjoint orbits correspond to $-\Im\Omega$ and $\Re\Omega$, respectively, and $\KK(x)$ sits in both orbits as a Lagrangian submanifold: 
 \[\xymatrix{ 
  & (\cO,\Omega)   &  \\
(X,\w=-\Im \Omega) \ar[ur]^-{}  &   &     (X^\vee,\Re\Omega)   \ar[ul]^-{}   \\  
 &  (\KK(x), 0) \ar[ur]^-{}  \ar[ul]^-{}  &
\qquad\qquad. }
\]
 If the flow for time $i\tfrac{\pi}{2}$ of a holomorphic vector field $\Lambda$ on (a subset of) $\cO$ is to intertwine
(subsets of) $(X,\sigma)$ and $(X^\vee,-\theta)$, it is natural that the linear projection $\gg^\C\to \kk^\C$ relates $\Lambda$ to (plus or minus) the Euler vector 
field. Furthermore, if the flow is to take $-\Im\Omega$ to $\Re\Omega$, then it is natural that $\Lambda$ be an anti-Liouville vector field for $\Omega$:
$\mathcal{L}_\Lambda \Omega=-\Omega$. The pullback form under the flow for time\footnote{A maximal integral curve of a holomorphic vector field has as domain a Riemann surface,
which maps to $\C$. For our integrated equations to make sense, we assume that we work with integral curves with (maybe small) domains inside $\C$.} 
$z\in \C$ is
$e^{-z}\Omega$,
and, therefore, for time $i\tfrac{\pi}{2}$ we obtain:
\[e^{-i\tfrac{\pi}{2}}\Omega=-i\Omega.\]

Because $X$ is a real form of $\cO$ and $-Y$ is a real analytic vector field on it, it
has a complexification $\Lambda$. In general, not much can be said about the domain of definition of a complexified object. However,
the following result --- whose proof is deferred to the Appendix --- shows that the domain
of definition of $\Lambda$ is rather large:
\begin{proposition}\label{pro:domain-complexification} The  vector field $\Lambda$ is defined in 
 \[\cB=\cA\cap -\theta(\cA)\subset \cO.\]
This is an open connected subset which is invariant under $\KK^\C$, and the restriction of $\Lambda$ to this subset  is $\KK^\C$-invariant.
\end{proposition}

Next, we address the relation of $\Lambda$ with the symplectic form $\Omega$:

\begin{lemma}\label{lem:hol-Liouville} 
 The  vector field $\Lambda$ on $\cB$ is an anti-Liouville vector field for $\Omega$:
 \begin{equation}\label{eq:Liouville}
  \mathcal{L}_\Lambda \Omega=-\Omega.
 \end{equation}
\end{lemma}
\begin{proof}
Because  (\ref{eq:Liouville}) is an equality of holomorphic 2-forms and $X^*$ is a real form of the connected $(\cB,-\theta)$, by analytic continuation
the equality holds if and only if
\[\mathcal{L}_\Lambda \Omega|_{X^*}=-\Omega|_{X^*}.\]
At a given point $y\in X^*$ we need to check an equality of complex linear 2-forms. Upon identifying $T^{1,0}_y\cO$ with $T_y\cO$, 
we are led to prove an equality of complex--valued $J$-complex 2-forms. The equality follows if it holds
for vectors on $T_yX^*$ because this is a real form of $(T_y\cO,J)$. Thus, it is enough to verify that the pullback of (\ref{eq:Liouville})
to $X^*$ holds. The pullback of the right--hand side is the purely imaginary 2-form $i\Im\Omega$;
this also implies that $\Omega$ equals the complexification of $-i\Im\Omega$. Therefore,
the left hand side  of (\ref{eq:Liouville}) is the complexification of $\mathcal{L}_{-Y}-i\Im\Omega$. Hence,  the pullback of (\ref{eq:Liouville}) 
to $X^*$ is
\[\mathcal{L}_Y i\Im\Omega=i\Im\Omega,\]
which holds by (\ref{eq:Kahler-potential}).
\end{proof}

We find it difficult to describe the properties of the flow of $\Omega$ on the whole $\cB$. However, we have a precise 
picture near the compact subset $\KK(x)$ of its zero set: 
\begin{lemma}\label{lem:C-star-action} There exists $\cB'\subset \cB$ a connected $\KK$-invariant open neighborhood of $\KK(x)$
where the (holomorphic) flow of $-\Lambda$ defines an action of the semigroup $\mathbb{D}^*\subset \C^*$ of complex numbers of norm
smaller that one.
 \end{lemma}
\begin{proof}
The complexification of the $\KK$-equivariant real analytic diffeomorphism (\ref{eq:canonical}) takes $-\Lambda$ to the holomorphic Euler vector
field of the complexification of the cotangent bundle $T^{*}\KK(x)$ (a holomorphic vector bundle).
Therefore, on a connected neighborhood $\cB'$ of $\KK(x)$ the vector field $\Lambda$ will share the following qualitative properties of 
the Euler vector field of a holomorphic vector bundle:
\begin{enumerate}[(i)]
 \item It vanishes along $\KK(x)$ and at any point there its intrinsic linearization is a (complex) projection. 
 \item Its flow at any point is defined in a neighborhood of the half plane $\{\Re z\leq 0\}$. 
 The restriction of the flow to $\{\Re z\leq 0\}$ integrates into an action of the semigroup  $\mathbb{D}^*\subset \C^*$.
\end{enumerate}
Finally, because $\KK$ is compact and the diffeomorphism (\ref{eq:canonical}) is $\KK$-equivariant, the connected neighborhood $\cB'$ can be assumed
to be $\KK$-invariant.
\end{proof}

\begin{remark} A (large) model for the complexification of $T^*\KK(x)$ is given by $T^{*1,0}\KK^\C(x)\overset{\chi}{\cong}\cA$ with involution
the cotangent lift of $-\theta$. It turns out that this cotangent lift is $-\theta$ itself acting on $\cA$. 
Therefore, we get an abstract identification 
of $\cB'\subset \cO$ with a subset of $\cA\subset\cO$. 
This identification is not given by the inclusion because $-\Lambda$ does not equal the Euler vector field of $\cA$.
(Lemma \ref{lem:Liouville-Hamiltonian}).
\end{remark}

We now look at how the action of the semigroup $\mathbb{D}^*$ described in Lemma \ref{lem:C-star-action} relates to the real forms
$X$ and $X^\vee$ and to the involutions.

\begin{proposition}\label{pro:action-exchange} The action of $\mathrm{e}^{-\tfrac{i\pi}{2}}\in \mathbb{D}^*$ on $\cB'$
in Lemma \ref{lem:C-star-action} interchanges 
the anti-holomorphic
 involutions $-\theta$ and $\sigma$ and therefore it interchanges their fixed point set $\cB'\cap X$ and $\cB'\cap X^\vee$.
\end{proposition}
\begin{proof} Let $y\in \cB'\cap X^*$ and let $z=\mathrm{e}^{-\tfrac{i\pi}{2}}\cdot y$. We can describe $z$ in a different manner: 
The point $y$ determines a backward 
trajectory of the real vector field $-\Re \Lambda=Y$ with limit point in $\KK(x)$. By $\KK$-invariance of $\Lambda$, we may assume
this point to be $x$. The tangent space to the unstable manifold of $Y$ 
at $x$ is $\Pi_*[i\kk,x]$ --- the orthogonal complement of $T_x\KK(x)$. By \cite[Theorem 3]{Na}, 
the point $y$ determines a unique vector $v\in \Pi_*[i\kk,x]$ such that the real analytic curve 
$y(t)$, $t\in [0,1]$, characterized by
\[ty'(t)=-\Re\Lambda(y(t)),\quad y'(0)=v,\]
satisfies $y(1)=y$. By items (i) and (ii) in the proof of Lemma \ref{lem:C-star-action} 
$z$ equals $z(1)$ for the real analytic curve $z(t)$, $t\in [0,1]$, characterized by:
\[tz'(t)=-\Re\Lambda(z(t)),\quad z'(0)=-Jv.\]
If we write $v=\Pi_*([iu,x])$, $u=\sum_{\alpha\in \Sigma^+(\ss)} u_\alpha+\theta u_\alpha$, $u_\alpha\in \gg_\alpha^\sigma$,
then by (\ref{eq:proj-tg})
\[-Jv= -iv=- i\sum_\alpha\alpha(x) i(u_\alpha+\theta u_\alpha)=-\sum_\alpha\alpha(x)(u_\alpha+\theta u_\alpha)\in \kk\subset \hh.\]
We want to show that $z(t)$ is contained in the real form $X^\vee=\mathrm{H}(x)$. 
Because  $Y$ is $\sigma$-invariant and $\cB$ is connected by analytic continuation  $-\Re\Lambda$ is also $\sigma$-invariant.
Because $\mathrm{H}(x)$ is the fixed point set of $\sigma$, the vector field  $-\Re\Lambda$ at points of the real form $\mathrm{H}(x)$ 
must be tangent to it. Therefore, its backward trajectories determine a submanifold of stable manifold at $x$  whose tangent space at $x$
is a subspace $W\subset T_x\mathrm{H}(x)\subset T_x \cA$. By a dimension count on $\mathrm{H}(x)$ it follows that $W$ must have half of the 
dimension of the tangent space of the stable manifold, which is a complex vector space. 
Because both $Jv$ and $W\subset T_x\mathrm{H}(x)$ are contained in the real form $\hh$, the subspace
spanned by them is totally real. Therefore, $Jv$ must be in $W$. Hence the trajectory $z(t)$ is in $X^\vee$ and so is its endpoint 
$z=\mathrm{e}^{-\tfrac{i\pi}{2}}\cdot y$.

The action of $\mathrm{e}^{-\tfrac{i\pi}{2}}$ on $\cB'\cap X^\vee$ is the action of
$\mathrm{e}^{-i\pi}=-1$ on $\cB'\cap X$. By appealing again to \cite[Theorem 3]{Na} or to Theorem \ref{thm:canonical}, the latter action is exactly that of the involution $\sigma$ on $\cB'\cap X$.

By the previous results, the conjugation of $-\theta$ and $\sigma$ by the action of $\mathrm{e}^{-\tfrac{i\pi}{2}}$ are anti-holomorphic involutions whose fixed point set is $\cB'\cap X^\vee$ and $\cB'\cap X$, respectively (and they are defined
on a connected neighborhood of their fixed point set). But
two anti-holomorphic involutions with equal fixed point set must be equal, as their composition is a holomorphic automorphism
which is the identity on a real form. Therefore, $-\theta$ and $\sigma$ are exchanged upon conjugation by the action of $\mathrm{e}^{-\tfrac{i\pi}{2}}$.
\end{proof}

\begin{remark}\label{rem:hol-disk} For any point $z\in \cB'$, the linear projection $\gg^\C\to \kk^\C$ 
identifies the compactification of the orbit $\mathbb{D}^*\cdot z$ 
with a holomorphic disk (perhaps of small radius). Under this identification,
$-\theta$ and $\sigma$ become the reflections on the real and imaginary axis, respectively. These holomorphic disks are the appropriate
non-linear lifts of the holomorphic disks associated to the Euler vector field of $\kk^\C$.
\end{remark}

We have all the ingredients to prove the symplectic correspondence between compact orbits and hyperbolic orbits:

\begin{proof}[Proof of Theorem \ref{thm:duality}]
 
The real part of  $\Lambda$ is tangent to both $X^*$ and $X^\vee$.  On $X^*$ it coincides
with $-Y$, which is complete for positive time. We shall argue that $\Re\Lambda$ is everywhere defined on $X^\vee$, and that it is complete there.

The vector field $\Lambda$ is defined on $\cB=\cA\cap -\theta(\cA)$. Let $z\in X^\vee=\cO^\sigma$. 
If $z\notin \cB$, then it belongs to an opposite Iwasawa fiber over a point in $X\backslash X^*$. Because the involution $\sigma$ preserves the Iwasawa
and opposite Iwasawa fibrations and acts freely on $X\backslash X^*$,
$\sigma(z)=z$ implies that $z$ would belong to two different fibers of the opposite 
Iwasawa fibration, which is not possible.  

By Theorem \ref{thm:canonical} and analytic continuation
the complex linear projection $\cB\subset \gg^\C\to \kk^\C$ relates $\Lambda$ to minus the holomorphic Euler vector field. The orbit 
$X^\vee$ is contained in the preimage of $\kk\subset \kk^\C$. Therefore,  the restriction of the linear projection  
$X^\vee\subset \kk\oplus \ss\to \kk$  relates $\Re\Lambda$ to minus the Euler vector field. Because the latter map is proper (it is the momentum
map for the action of $\KK$)
and the Euler vector field
is complete, we conclude that $\Re\Lambda|_{X^\vee}$ is complete. 

We now proceed to define the map $\Psi$ in (\ref{eq:duality}).  Given $y\in X^*$ because $-Y$ is complete for positive time,
there exists $t_y\geq 0$  such that the flow of $\Re\Lambda$ for any time greater than $t_y$ takes
$y$ into $\cB'\cap X^*$. We choose any $t>t_y$ and apply the flow map of $\Re\Lambda$. Next,  since we are in $\cB'$, 
we can let $-i\tfrac{\pi}{2}$ act on this point as defined in Lemma \ref{lem:C-star-action}; this action is the flow map of $\Im\Lambda$
for time $i\tfrac{\pi}{2}$. Finally, because by Proposition \ref{pro:action-exchange} the resulting point is in $X^\vee$, by the previous
paragraph we can apply to 
it the flow map of $\Re\Lambda$ for time $-t$. 

The point $\Psi(y)$ does not depend on the choice of $t>t_y$. By Lemma \ref{lem:C-star-action}
at points in $\cB'$, the flow of $\Lambda$ is defined in the positive half plane (and thus the flows of $\Re\Lambda$ and $\Im\Lambda$ for the corresponding
times commute). In particular, at $y\in \cB'$ the definition of $\Psi$ is the Wick rotation given by the flow of $\Lambda$ for time $i\tfrac{\pi}{2}$.
By elementary ordinary differential equations (ODEs) theory in a neighborhood of a fixed $y\in X^*$, we can take a common time $t>0$. This implies that $\Psi$ is a real analytic
local diffeomorphism. If $y,y'\in X^*$ are different points, then we can always find a common flow time $t>0$. Then, $\Psi$ for both points becomes
the composition of the same three injective maps, and thus the images differ. If we denote by $D^\vee\subset X^\vee$ the image of $\Psi$,
we conclude that $\Psi:X^*\to D^\vee$ is a real analytic diffeomorphism.

By Proposition \ref{pro:action-exchange} the flow of $\Lambda$ for time $i\tfrac{\pi}{2}$ on $\cB'$ takes $\sigma$  to $-\theta$.
Because $\Psi$, $\sigma$, and $-\theta$ are real analytic, by analytic continuation $\Psi:X^*\to D^\vee$ must intertwine the involutions everywhere.

By Lemma \ref{lem:hol-Liouville} $\Lambda$ is anti-Liouville for $\Omega$. Therefore, on $\cB'$ its  flow for time $i\tfrac{\pi}{2}$  pulls backs 
$\Omega$ to $-i\Omega$, and thus $\Re\Omega$ to $-\Im\Omega$. Because $\Psi$ is a real analytic map and $\Omega$ is a holomorphic 2-form, once more by analytic continuation, $\Psi:X^*\to D^\vee$ must pull back $\Re\Omega$ to $-\Im\Omega$ everywhere.  
\end{proof}

\begin{example}\label{ex:sl(2C)} We let $x\in X\subset \SL(2,\C)$ as in Example \ref{ex:su(2)}.
The points in $X\backslash X^*$ are 
\[\pm \begin{pmatrix} 0 & -i\\ i & 0\end{pmatrix}=\frac{\sqrt{2}}{2}\begin{pmatrix} 1 & \pm i\\ \pm i & 1\end{pmatrix}\cdot x.\]
Hence, the vector part of the positive and negative Iwasawa fibers over them are:
\begin{equation}\label{eq:Iwasawa-fibers}
 \zeta\left(\begin{pmatrix} 0 & 1\\ 1 & 0\end{pmatrix}\mp \begin{pmatrix} 1 & 0\\ 0 & -1\end{pmatrix}\right),\,\,\zeta\in \C,
\quad\ 
 \zeta\left(\begin{pmatrix} 0 & 1\\ 1 & 0\end{pmatrix}\pm \begin{pmatrix} 1 & 0\\ 0 & -1\end{pmatrix}\right)\,\,\zeta\in \C.
\end{equation}
Let  $e,f,z$ be coordinates on $i\mathfrak{su}(2)=\pp\oplus i\kk$ in the basis of (cyclically) permuted Pauli matrices and 
let $E,F,Z$ be the complexified coordinates on $\mathfrak{sl}(2,\C)$. By (\ref{eq:Iwasawa-fibers}), the vector part of the
positive and negative Iwasawa fibers over $X\backslash X^*$
are the two lines of the quadric $\{E^2+F^2=0\}\subset \pp\oplus i\pp$.

The parameterization (\ref{eq:hyp-coord}) sends $\frac{1}{2}\sinh(2b)\cosh(2b)\frac{\partial}{\partial b}$ to
\[Y=-\frac{ez^2}{e^2+f^2}\frac{\partial}{\partial e}-\frac{fz^2}{e^2+f^2}\frac{\partial}{\partial f}+z\frac{\partial}{\partial z}.\]
The formula above is valid for every $\SU(2)$-orbit in $i\mathfrak{su}(2)$. The complexification of $Y$ is the vector field 
\begin{equation}\label{eq:complexified-flow-su(2)}
-\Lambda=-\frac{EZ^2}{E^2+F^2}\frac{\partial}{\partial E}-\frac{FZ^2}{E^2+F^2}\frac{\partial}{\partial F}+Z\frac{\partial}{\partial Z}.
\end{equation}
It is defined away from the locus $\{E^2+F^2=0\}\subset \mathfrak{sl}(2,\C)$. The points of $X\backslash X^*$ correspond to  $Z=\pm 1$.
Therefore,  on $\cO=\SL(2,\C)(x)$ the vector field $-\Lambda$ is defined exactly 
in the complement of the Iwasawa and opposite Iwasawa fibers over $X\backslash X^*$. This is the subset $\cB\subset \cO$,
so the result is consistent with Proposition \ref{pro:domain-complexification}.

Equation (\ref{eq:complexified-flow-su(2)}) yields elementary O.D.E.'s for the third component of the flow
and for the sum of the squares of the first and second components of the flow. Thus,  the flow of  $-\Lambda$ for time $w\in \C$ is given by
\[\Phi_w=\left(\left(1+\frac{Z^2}{E^2+F^2}(1-\mathrm{e}^{2w})\right)^{1/2}E,
\left(1+\frac{Z^2}{E^2+F^2}(1-\mathrm{e}^{2w})\right)^{1/2}F,\mathrm{e}^wZ
\right),\]
where the branches of the square root in the first and second components are the standard ones. The flow preserves the Killing form
of $\mathfrak{sl}(2,\C)$. Thus, it evolves along adjoint orbits. 
The trajectories starting at $X^*\subset \cB$ correspond to real values $(e,f,z)$. Because the equation 
\[1+\frac{z^2}{e^2+f^2}(1-\mathrm{e}^{2w})=0\]
has no solutions for $\Re w\leq 0$, we conclude that the flow $\Phi$ at points in $X^*$ is defined in the negative half plane $\{\Re w\leq 0\}$.
This is consistent with Lemma \ref{lem:C-star-action} and Proposition \ref{pro:action-exchange}. In fact in this case $\Phi$ is the Wick-type rotation,
as Lemma \ref{lem:C-star-action} is valid in the whole open sense subset $\cB\subset \cO$. The diffeomorphism
\[\Psi(e,f,z)=\Phi_{-i\tfrac{\pi}{2}}(e,f,z)=\left(\left(1+\frac{2z^2}{e^2+f^2}\right)^{1/2}e,
\left(1+\frac{2z^2}{e^2+f^2}\right)^{1/2}f,-iz\right)\]
takes $i\mathfrak{su}(2)=\pp\oplus i\kk$ to $\mathfrak{sl}(2,\R)=\pp\oplus \kk$ therefore sending 
$X^*\subset \{e^2+f^2+z^2=1\}$  into the hyperboloid 
$\{e^2+f^2+(iz^2)=1\}=\SL(2,\R)(x)\subset \mathfrak{sl}(2,\R)$.
One may rewrite 
\begin{equation}\label{eq:spherical-to-hyperbolic}
 \Psi(e,0,z)=(\cosh(\ln(z)+\sqrt{z^2+1}),0,z).
\end{equation}

\end{example}
 
\begin{remark} The real hyperbolic orbit $X^\vee$ of the non-compact semisimple Lie algebra $\hh$ is canonically symplectomorphic to 
the cotangent bundle over the Lagrangian $L$, which is its submanifold of real flags 
(the intersection with the -1-eigen-space in the Cartan decomposition) \cite{Mt}:
\[\Xi: (X^\vee, \Re\Omega)\to (T^*L,d\lambda).\]
The Iwasawa ruling together with twice the Killing form identify the orbit with the cotangent bundle of the submanifold of real flags. 
With this identification, the symplectomorphism $\Xi$ is the identity. The reason is that the Euler vector field
is a (complete) Liouville vector field for the KKS symplectic form.

The composition of the symplectomorphism $\psi^{-1}$ and $\Psi$ in Theorems \ref{thm:canonical} and \ref{thm:duality} with $\Xi $
is a $\KK$-equivariant symplectomorphism
\[\Upsilon:(D\subset T^*L,d\lambda)\to (\Psi(D^\vee)\subset T^*L,d\lambda)\]
which is the identity on the zero section. 
In the case of $X\subset i\mathfrak{su}(2)$ discussed in Examples \ref{ex:su(2)} and  \ref{ex:sl(2C)} we have
 $\Upsilon(D)=D$. However, the corresponding automorphism $\Upsilon$ is not the identity. Equation (\ref{eq:spherical-to-hyperbolic})
shows that a trajectory of the vector field $Y$ is not sent to a trajectory of the Euler vector field (the image of the trajectory is not in a Iwasawa
fiber of the hyperboloid).
\end{remark}

\section{Appendix: The domain of the complexified Liouville vector field}\label{sec:appendix}

For a given symplectic form, the set of Liouville vector fields is  an affine space whose vector space are symplectic vector fields.
If we denote by $\varXi$  the Euler vector field of $\cA$ (rather, the image of the Euler vector field by $\chi$),
then the difference $\varXi-\Lambda$ must be a symplectic vector field. The diffeomorphism
given by the Iwasawa projection identifies the closed 1-form $\beta\in \Omega^1(\KK^\C(x))$
with a real analytic 1-form on $X^*$, which we still denote by $\beta$. Because the latter is an open subset of a real form for $(\cO,-\theta)$,
it admits a complexification $\beta^\C$, which determines a symplectic vector field $X_{\beta^\C}$:
\[\iota_{X_{\beta^\C}}\Omega=\beta^\C.\]

\begin{lemma}\label{lem:Liouville-Hamiltonian} The  difference of the Euler vector field and $\Lambda$ 
is the symplectic vector field determined by $\beta^\C$:
\begin{equation}\label{eq:Liouville-Hamiltonian}
 \varXi-\Lambda=X_{\beta^\C}.
\end{equation}
The equality is valid on any connected open neighborhood of $X^*$ in $\cA$. 
 \end{lemma}
\begin{proof}
 Because $X^*\subset X$ is an open subset of a real form of $\cO$, it suffices to prove (\ref{eq:Liouville-Hamiltonian})
 the equality in points of  $X^*$:
 \[ (\varXi-\Lambda)|_{X^*}={X_{\beta^\C}}|_{X^*}.\]
 This equality is equivalent to the one obtained by taking contraction with $\Omega$,
 \[\iota_{\varXi-\Lambda}\Omega|_{X^*}=\iota_{X_{\beta^\C}}\Omega|_{X^*},\]
 which leads us to proving the equality of holomorphic 1-forms:
 \[\lambda|_{X^*}-\iota_{\Lambda}\Omega|_{X^*}=\beta^\C|_{X^*},\]
 where $\lambda$ is the tautological 1-form of the cotangent bundle. It is enough to test the 1-forms on 
 vectors tangent to $X^*$.
 At a point $y\in X^*$, the right--hand side returns $\beta_y$.
 The expression of the tautological 1-form at $y\in X^*$ is obtained by regarding the point as a complex 1-form  
 via the identification of  $\cA$ with $T^{*1,0}\KK^\C(x)$:
 \[\lambda_y=\beta_y-i\beta_y\circ J.\] 
 Because the tautological 1-form vanishes along fibers of the cotangent bundle, we have:
 \[\lambda|_{T_yX^*}=\beta_y-i\beta_y\circ \Pi_*\circ J.\]
 Because $\iota_\Lambda \Omega$ is the complexification of $\iota_Y-i\Im\Omega$:
 \[\iota_\Lambda\Omega|_{T_yX^*}=-\iota_Yi\Im\Omega|_{T_yX^*}=i\iota_Y-\Im\Omega=-i\beta_y\circ \Pi_*\circ J.\]
 Hence, the left--hand side gives:
 \[\lambda|_{T_yX^*}-\iota_\Lambda\Omega|_{T_yX^*}=\beta_y-i\beta_y\circ \Pi_*\circ J+i\beta_y\circ \Pi_*\circ J
 =\beta_y,\] 
 which proves the equality.
\end{proof} 

By (\ref{eq:Liouville-Hamiltonian}) the (connected) domain of definition  of $\Lambda$ and $\beta^\C$ in $\cA$ is the same.
The 1-form $\beta\in \Omega^1(\KK^\C(x))$ is constructed geometrically as the section of $\Pi:T^*\KK^\C(x)\to \KK^\C(x)$ determined by $X^*$. 
Our purpose is to complexify the previous geometric construction.

Let $\gg^\mathbb{H}=\gg^\C\oplus j\gg^\C$ be the complexification of  $\gg^\C$ and let 
$\GG^\mathbb{H}$
denote its simply connected integration.
There is an isomorphism of complex Lie algebras:
\begin{equation}\label{eq:iso-quat}
f_1:\gg^\mathbb{H}\to \gg^\C\times \gg^\C,\quad u+jv\mapsto (u+iv,\theta(u-iv)). 
\end{equation}
More precisely: $u_\gg+u_{i\gg}+jv_\gg+jv_{i\gg}\mapsto (u_\gg+iv_{i\gg}+iv_{\gg}+u_{i\gg},u_\gg-iv_{i\gg}-u_{i\gg}+iv_{\gg})$.
On $\gg^\mathbb{H}$ the complexifications $\theta^\C$ and $\sigma^\C$  are commuting holomorphic involutions.
Let $\tilde{x}=-jix\in j\gg$. Because it is fixed by $\theta^\C$ and $-\sigma^\C$
the involutions act on the orbit $\GG^\mathbb{H}(\tilde{x})$, which we denote by $\cO^\mathbb{H}$. Their respective fixed point sets
are better understood by looking at their images by $f_1$ in the product Lie algebra. To do this analysis, we need to introduce
subalgebras, subgroups, and orbits on both sides of (\ref{eq:iso-quat}).
As for the left--hand side,  we denote by $\KK^\C$, $\GG^\C$, and $\KK^\mathbb{H}$
the complex subgroups of $\GG^{\mathbb{H}}$ that integrate the subalgebras 
\begin{equation}\label{eq:many-subalgebras}
  \kk\oplus j\kk,\quad \gg\oplus j\gg, \quad \kk\oplus i\kk\oplus j\kk\oplus ji\kk.
\end{equation}
 In the product Lie algebra $\gg^\C\times \gg^\C$ and in its product integration,
we use the subindex $\Delta$ to refer to diagonal subalgebras and subgroups; the subindex $\Delta^\theta$ describes the image of diagonal subalgebras
and subgroups by the involution that is the identity on the first factor and $\theta$ on the second factor.
We use the same notation for the Lie algebra isomorphism $f_1$ 
and for its integration.

The following lemma contains straightforward computations: 
\begin{lemma}\label{lem:intertwine} The isomorphism $f_1$ in (\ref{eq:iso-quat}) has the following properties: 
\begin{enumerate}[(i)]
 \item It identifies the subgroups  $\KK^\C$, $\GG^\C$ and $\KK^\mathbb{H}$ with $\KK^\C_\Delta$, $\GG^\C_\Delta$, and $\KK^\C\times \KK^\C$.
 \item It takes $\cO^\mathbb{H}$ to the $\GG^\C\times \GG^\C$-orbit of $(x,x)$.
  \item It intertwines  $\theta^\C$ and the transposition of factors.
\end{enumerate} 
\end{lemma}

On the semisimple product Lie algebra and Lie group, we fix the `anti-diagonal' Iwasawa decomposition:
\[\GG^\C\times \GG^\C=\GG \mathrm{AN}\times \GG \mathrm{A}\overline{\mathrm{N}}.\]
Via the isomorphism $f_1$  in (\ref{eq:iso-quat}), we induce an Iwasawa decomposition for $\GG^\mathbb{H}$ 
whose nilpotent Lie algebra we denote by $\nn^\mathbb{H}$.
By item (ii) in Lemma \ref{lem:intertwine} $\tilde{x}$
corresponds to a hyperbolic element. Thus, $\cO^\mathbb{H}$ supports the corresponding Iwasawa ruling (and its opposite one).

\begin{lemma}\label{lem:fixedpoints-quat1} The fixed point set of $-\theta^\C\sigma^\C$  on $\cO^\mathbb{H}$ is the orbit
$\KK^\mathbb{H}(\tilde{x})$ and it is a Lagrangian section of the Iwasawa rulings. 
\end{lemma}
\begin{proof} The fixed point set of $\theta^\C\sigma^\C$ on $\GG^\mathbb{H}$ is a spherical subgroup and  $\tilde{x}$ is a hyperbolic element fixed by 
$-\theta^\C\sigma^\C$. Therefore, 
the result is proved exactly as item (ii) in Proposition \ref{pro:fixedpoints-sections} complemented
by item (4) in Section \ref{ssec:involutions}.
\end{proof}

The isomorphism of complex Lie algebras 
\begin{equation}\label{iso-complex}
      f_2:\gg^\C\to \gg\oplus j\gg\subset \gg^\mathbb{H},\quad u+iv\mapsto u+jv,                                   
\end{equation}
maps $x$ to $\tilde{x}$ and restricts to an equivariant biholomorphism
from $\cO$ to $\GG^\C(\tilde{x})\subset \cO^\mathbb{H}$.

\begin{lemma}\label{lem:fixedpoints-quat2} 
The fixed point set of $\theta^\C$  on $\cO^\mathbb{H}$
is the orbit $\GG^\C(\tilde{x})$ and it is a section of the Iwasawa ruling. The subset of points of $\GG^\C(\tilde{x})$ that the Iwasawa map 
relates to points in the Lagrangian $\KK^\mathbb{H}(\tilde{x})$ equals $f_2(\cB)$.
\end{lemma}
\begin{proof}
The involution  $\theta^\C$ fixes $\tilde{x}$ so it acts on $\cO^\mathbb{H}$. 
We use the linear isomorphism $f_1$ to take the problem of describing its fixed point set to the product Lie algebra.
By items (ii) and  (iii) in Lemma \ref{lem:intertwine} the fixed point set
in  $\gg^\C\times \gg^\C$ is the intersection of the orbit $\GG^\C\times \GG^\C(x,x)$ with the diagonal. This is exactly
the orbit by the diagonal subgroup $ \GG^\C_\Delta$. Therefore, by item (i) in Lemma \ref{lem:intertwine}
\[{(\cO^\mathbb{H})}^{\theta^\C}=\GG^\C(\tilde{x}).\]

At $(x,x)$ the tangent space
of the orbit $\GG^\C_\Delta(x,x)$ and of the Iwasawa fiber are the diagonal subalgebra of the sum of the tangent spaces to the two Iwasawa fibers through $x$
and the product of both tangent spaces, respectively:
\[(\sum_{\alpha(x)>0} \gg_\alpha\oplus \gg_{-\alpha})_\Delta,\quad  \sum_{\alpha(x)>0} \gg_\alpha\times \sum_{\alpha(x)>0} \gg_{-\alpha}.\]
They are complementary subspaces. Because all of the structure is invariant by the action of  $\GG^\C\subset \GG^\mathbb{H}$, transversality holds at every point  of  $\GG^\C_\Delta(x,x)$.
Because $\GG^\C(\tilde{x})$ is the intersection of $\cO^\mathbb{H}$ with a vector subspace 
if two points of the same Iwasawa fiber intersected $\GG^\C(\tilde{x})$, then the line joining both points should also be 
in $\GG^\C(\tilde{x})$. This would contradict the transversality of the intersection. Therefore, $\GG^\C(\tilde{x})$ is a section 
of the Iwasawa ruling.

Let $z\in \cO^\mathbb{H}$ be a point in the image by $\KK^\mathbb{H}$ of the Iwasawa fiber over $-jix$. Its image $f_1(z)$ can be written as:
\[(h\cdot (x+u),h'\cdot (x+u')),\quad  h,h'\in \KK^\C,\, u\in \sum_{\alpha(x)>0} \gg_\alpha,\,u'\in \sum_{\alpha(x)>0} \gg_{-\alpha}.\]
The point $f_1(z)$ belongs to  $\GG^\C_\Delta(x,x)$ if and only if for some $g\in \GG^\C$, we have
\[(h\cdot (x+u),h'\cdot (x+u'))=(g\cdot x,g\cdot x).\]
If we write $x+u=l\cdot u$, $l\in \mathrm{N}$, $x+u'=l'\cdot x$, $l'\in \overline{\mathrm{N}}$, then this is equivalent to 
\[hl\cdot x=g\cdot x,\quad hl'\cdot x=g\cdot x,\]
which can be rewritten as $g\cdot x\in \cA\cap -\theta(\cA)=\cB$.

Because the composition $f_1\circ f_2$ is the diagonal embedding, we conclude that the Iwasawa map on $\cO^\mathbb{H}$ relates 
$\KK^\mathbb{H}(\tilde{x})$ with $f_2(\cB)\subset \GG^\C(\tilde{x})$.
 \end{proof}

To discuss why the orbits we constructed are complexifications of the orbits that are involved in the definition of $\beta$,
we introduce  $t$ the anti-holomorphic involution on $\gg^\mathbb{H}$ given by conjugation with respect to $j$. 
The point $\tilde{x}$ belongs to the fixed point set of $-t$.
 
\begin{lemma}\label{lem:complexifications} 
The anti-holomorphic involution $-t$ acts on $\cO^{\mathbb{H}}$ in an anti-complex affine fashion and it preserves the open subset 
 $\KK^\mathbb{H}(\tilde{x}+\nn^\mathbb{H})$. The isomorphism  of Lie algebras $f_1$ (\ref{eq:iso-quat}) 
 takes the real forms with respect to $-t$ 
 of  $\KK^\mathbb{H}(\tilde{x})$, $\GG^\C(\tilde{x})$, and $\KK^\mathbb{H}(\tilde{x}+\nn^\mathbb{H})$ to 
 $\KK^\C_{\Delta^\theta}(x,x)$, $\GG_\Delta(x,x)$, and $\KK^\C_{\Delta^\theta}\mathrm{N}_{\Delta^\theta}(x,x)$, respectively.
\end{lemma}
\begin{proof}
The linear isomorphism $f_1$ takes $t$ to the transposition of factors followed by $\theta$ on both factors. 
Therefore, $t$ on the product Lie algebra sends the
Iwasawa fiber $\nn(x)\times\overline{\nn}(x)$ to itself in an anti-complex fashion. It also sends the maximal compact subgroup $\GG\times \GG$ and 
 $\KK^\C\times \KK^\C$ to themselves. 
Hence, $-t$ acts by anti-complex affine transformations on $\cO^{\mathbb{H}}$ preserving $\KK^\mathbb{H}(\tilde{x}+\nn^\mathbb{H})$.

Because the involutions $t,\theta^\C$, and $\sigma^\C$ commute, $-t$ acts on
$\KK^\mathbb{H}(\tilde{x})$ and $\GG^\C(\tilde{x})$. By (i) in Lemma \ref{lem:intertwine}
the image by $f_1$ of these orbits are the orbits $\GG^\C_\Delta(x,x)$ and $\KK^\C\times \KK^\C(x,x)$. Therefore, 
the images by $f_1$ of their real forms with respect to $-t$ equal the
orbits  $\KK^\C_{\Delta^\theta}(x,x)$ and $\GG_{\Delta}(x,x)$, respectively. Likewise, the image by $f_1$
of the real form of $\cO^\mathbb{H}$ must contain the orbit $\GG^\C_{\Delta^\theta}(x,x)$. Its open subset 
$\KK^\C_{\Delta^\theta}\mathrm{N}_{\Delta^\theta}(x,x)$ is contained in $\KK^\mathbb{H}(\tilde{x}+\nn^\mathbb{H})$ and thus on its 
real form. Because $-t$ acts by bundle isomorphism preserving the section $\KK^\mathbb{H}(\tilde{x})$, 
the real form must be a subbundle over its fixed point set. Hence, the image by $f_1$ of the real form of $\KK^\mathbb{H}(\tilde{x}+\nn^\mathbb{H})$
must be 
\begin{equation}\label{eq:fixed-point-t}
\KK^\C_{\Delta^\theta}\mathrm{N}_{\Delta^\theta}(x,x)=\KK^\C_{\Delta^\theta}(x+m,x-\theta m),\quad m\in \nn(x).
\end{equation}
\end{proof}

By Lemma \ref{lem:fixedpoints-quat1} $\KK^\mathbb{H}(\tilde{x})$ is a Lagrangian submanifold of the orbit $\cO^\mathbb{H}$
which is a section of the Iwasawa ruling. Therefore, by Proposition \ref{pro:hol-ctg}, the open subset $\KK^\mathbb{H}(\tilde{x}+\nn^\mathbb{H})$
is canonically identified with the holomorphic cotangent bundle of $\KK^\mathbb{H}(\tilde{x})$; 
in fact the statement about the symplectic form is not needed here and all is needed is the use of the Killing form.
By item (i) in Lemma \ref{lem:intertwine} the linear isomorphism $f_1$ takes the vector bundle $\KK^\mathbb{H}(\tilde{x}+\nn^\mathbb{H})$
to the product vector bundle
$\KK^\C(x)(x+\nn(x))\times \KK^\C(x+\overline{\nn}(x))$. It also takes the Killing form on $\gg^\mathbb{H}$ to the product Killing form. Hence, 
$f_1$ takes the cotangent bundle structure of $\KK^\mathbb{H}(\tilde{x}+\nn^\mathbb{H})$ to the 
product cotangent bundle structure on $\KK^\C(x)(x+\nn(x))\times \KK^\C(x+\overline{\nn}(x))$.
We shall modify the latter identification by composing with bundle automorphisms of each factor. Specifically, we 
will fix the identification 
\begin{equation}\label{eq:rescaled-ctg}
T^{*1,0}\KK^\C(x)\times T^{*1,0}\KK^\C(x)\overset{\chi_1}{\cong}  \KK^\C(x)(x+\nn(x))\times \KK^\C(x+\overline{\nn}(x)),
 \end{equation}
which sends a vector in the fibers to a covector by means of one half of the Killing form in the first factor and minus one half
of the Killing form in the second factor. Via $f_1$ this is transferred to another identification 
\begin{equation}\label{eq:rescaled-ctg2}
 T^{*1,0}\KK^\C(\tilde{x}) \overset{\chi^\mathbb{H}}{\cong} \KK^\mathbb{H}(\tilde{x}).
 \end{equation}
 
We define the real analytic diffeomorphism 
\begin{equation}\label{eq:realform-identification}
 f_3: \cA\to \KK^\C_{\Delta^\theta}\mathrm{N}_{\Delta^\theta}(x,x),\quad hl\cdot x\mapsto (hl\cdot x,\theta(hl)\cdot x).
\end{equation}

\begin{lemma}\label{lem:twisted-ctg} The identification $\chi_1$ in (\ref{eq:rescaled-ctg}) has the following properties:
\begin{enumerate}[(i)]
 \item It takes the real form  $\KK^\C_{\Delta^\theta}\mathrm{N}_{\Delta^\theta}(x,x)$
to the cotangent bundle of $\KK^\C_{\Delta^\theta}(x,x)$, which is the real form of the zero section. 
Therefore, the holomorphic cotangent bundle on the right--hand side 
of (\ref{eq:rescaled-ctg}) is the complexification of the cotangent bundle of the real form 
$\KK^\C_{\Delta^\theta}(x,x)\subset \KK^\C\times \KK^\C(x,x)$.
\item Its composition on the right with  $T^{*}\KK^\C(x)\times T^{*}\KK^\C(x)\equiv T^{*1,0}\KK^\C(x)\times T^{*1,0}\KK^\C(x)$ and on the left with 
$f_3^{-1}$ is the identification
\[T^*\KK^\C(x)\equiv T^{*1,0}\KK^\C(x)\overset{\chi}{\cong}\cA.\]
\end{enumerate}
\end{lemma}
\begin{proof}
By (\ref{eq:fixed-point-t}), a point $z$ in the real form of $\KK^\C(x)(x+\nn(x))\times \KK^\C(x+\overline{\nn}(x))$ can be written:
\[z=(hl\cdot x,\theta(hl)\cdot x)=(h\cdot(x+m),\theta(h)\cdot (x-\theta m)),\quad h\in \KK^\C,\,\,l\in \mathrm{N},\,\,m\in \nn(x)\in \overline{\mathrm{N}}.\]
The identification $\chi_1$ takes $z$ to a complex covector. We shall take its real part to work in the ordinary cotangent bundle. The action on
the tangent vector $(h\cdot [u,x],\theta(h)\cdot [v,x])$, $u,v\in \kk^\C$, is given by 
\begin{equation}\label{eq:real-covecotr}
\frac{1}{2}\Re\langle m,u \rangle+\frac{1}{2}\Re\langle\theta(h)\cdot \theta m,\theta(h)\cdot v\rangle
=\frac{1}{2}\Re\langle m,u \rangle+\frac{1}{2}\Re\langle \theta m,v\rangle,
\end{equation} 
where we used that  $\theta$ is an automorphism of the real part of the Killing form.
 By Lemma \ref{lem:complexifications}, the real form of the zero section is $\KK^\C_{\Delta^\theta}(x,x)\subset \KK^\C\times \KK^\C(x,x)$. Therefore, vectors in the
$\pm 1$-eigen-bundles are, respectively, of the form
\[(h_\cdot [u,x],\theta(h)\cdot[\theta(u),x]),\quad (h\cdot [u,x],\theta(h)\cdot[-\theta(u),x]).\]
Equation (\ref{eq:real-covecotr})
implies that $z$ annihilates the $-1$-eigen-space, which proves item (i). A posteriori we deduce that the involution on the right--hand side 
of (\ref{eq:rescaled-ctg}) is the cotangent lift of the involution on the zero section (the cotangent lift is naturally
defined in the real cotangent bundle). 

A point $hl\cdot x=h\cdot (x=m)\in \cA$ corresponds by $\chi$ in (\ref{eq:holomorphic-ctg-bundle}) to a real covector whose 
action on the vector $h_*[u,x]$ is $\Re\langle m,u\rangle $. Because $f_3$ sends the point and vector to 
\[(hl\cdot x,\theta(hl)\cdot x)=(h\cdot(x+m),\theta(h)\cdot (x-\theta m)),\quad (h_\cdot [u,x],\theta(h)\cdot[\theta(u),x]),\]
Equation (\ref{eq:real-covecotr}) shows that $f_3$ is compatible with the cotangent bundle identifications $\chi$ and $\chi_1$.
\end{proof}

\begin{proof}[Proof of Proposition \ref{pro:domain-complexification}]

The equality (\ref{eq:Liouville-Hamiltonian}) implies that discussing domains of definition of $\Lambda$ inside of $\cA$
is the same as discussing domains of definition of the complexification of $\beta\in \Omega^1(X^*)$. 

By Lemmas \ref{lem:fixedpoints-quat1} and \ref{lem:fixedpoints-quat2},  the affine bundle $\KK^\mathbb{H}(\tilde{x})$
has sections  $\KK^\mathbb{H}(\tilde{x})$ and $\GG^\C(\tilde{x})$. The former section is used as zero section of a vector bundle 
structure and the latter is seen as a section of this bundle. By Lemma \ref{lem:fixedpoints-quat2} only the open subset $f_2(\cB)\subset \GG^\C(\tilde{x})$
defines a section. Because  $\chi^\mathbb{H}$ in (\ref{eq:rescaled-ctg2}) is an identification of vector bundles, the section $f_2(\cB)$ is identified
with a 1-form $\Sigma'\in\Omega^{1,0}(\KK^\mathbb{H}(\tilde{x}))$. The Iwasawa projection identifies it with a 1-form 
$\Sigma\in \Omega^{1,0}(f_2(\cB))$.

By Lemma \ref{lem:complexifications}, the linear isomorphism $f_1$ takes the real form with respect to $-t$ 
of the affine bundle and both sections to  $\KK^\C_{\Delta^\theta}(x,x)$, $\GG_\Delta(x,x)$ and $\KK^\C_{\Delta^\theta}\mathrm{N}_{\Delta^\theta}(x,x)$, respectively.
Strictly speaking, the lemma discusses the real form for the whole section $\GG^\C(\tilde{x})$. Because 
$f_2(\cB)=\GG^\C(\tilde{x})\cap \KK^\mathbb{H}(\tilde{x}+\nn^\mathbb{H})$ and the involution preserves the affine fibers, the real form
of $f_2(\cB)$ is sent to 
\[\GG_\Delta(x,x)\cap \KK^\C(x+\nn)\times \KK^\C(x+\overline{\nn})=X^*_\Delta.\]
Item (i) in Lemma \ref{lem:twisted-ctg} implies that $\Sigma'$ is the complexification of 
\[\Re\Sigma'|_ {\KK^\mathbb{H}(\tilde{x})^{-t}}\in \Omega^1(\KK^\mathbb{H}(\tilde{x})^{-t}),\]
and the same holds for $\Sigma$.
By item (ii) in \ref{lem:twisted-ctg}
\[(f_1^{-1}\circ f_3)^*\Re\Sigma=\beta\in \Omega^1(X^*).\]
 Because the composition $f_1\circ f_2$ is the diagonal embedding
 \[f_1\circ f_2|_{X^*}=f_3.\]
 The conclusion is that 
 \[f_2^*\Sigma\in\Omega^{1,0}(\cB),\quad f_2^*\Re\Sigma|_{X^*}=\beta\in \Omega^1(X^*),\]
 and therefore the complexification of $\beta$ is defined in $\cB$. 
 
To discuss topological properties of $\cB$ we use the viewpoint of complex homogeneous spaces on  (\ref{eq:quasiproj-orbit}).
We may take $\GG^\C$ to be a linear algebraic group. Then the Iwasawa map is a morphism from an algebraic to a projective variety.
Because $\KK^\C$ is an algebraic group $X^*$ is a Zariski open subset, and so is its inverse image $\cA$. The reasoning for $-\theta(\cA)$ is the same
once we replace the parabolic subgroup $\mathrm{P}$ by its opposite $\theta(\mathrm{P})$. Therefore, $\cB$ is a Zariski open subset and in particular  it is connected.

By construction $\cA$ is saturated by $\KK^\C$-orbits. The same applies to $-\theta(\cA)=\KK^\C(x+\overline{\nn})$. Hence
$\cB$ is saturated by orbits of $\KK^\C$.

Because $Y$ is $\KK$-invariant and $\cB$ is connected and $\KK^\C$-invariant, by analytic continuation, $\Lambda$ is also $\KK^\C$-invariant.
\end{proof}






\end{document}